\RequirePackage{amsthm}\RequirePackage{thmtools,thm-restate}
\documentclass[thm-restate, cleveref, numberwithinsect, UKenglish]{lipics-v2019}

\newif\iflongversion
\longversiontrue

\iflongversion
  \nolinenumbers
\fi
\usepackage{amsmath}
\usepackage{amsthm}
\usepackage{amssymb}
\usepackage{xparse}
\usepackage{tocloft}
\usepackage{graphicx}
\usepackage{stmaryrd}
\usepackage{tikz}
\usepackage[obeyFinal]{todonotes}
\usepackage{mathtools}
\usepackage{xspace}
\usepackage{subcaption}
\usetikzlibrary{automata}
\usetikzlibrary{arrows}

\usetikzlibrary{shapes,positioning,fit,calc,decorations.text,decorations.pathmorphing,backgrounds}
\usetikzlibrary{arrows,decorations.pathreplacing}

\DeclareDocumentCommand{\runs}{O{} O{} m}{%
  \mathsf{Runs}_{#1}^{#2}(#3)
}

\def\pe{\mathsf{pe}}
\def\thickness{\mathsf{thickness}}

\DeclareDocumentCommand{\retruns}{O{} O{} m}{%
  \mathsf{Ret}_{#1}^{#2}(#3)
}

\DeclareDocumentCommand{\leftruns}{O{} O{} m}{%
  \mathsf{RetL}_{#1}^{#2}(#3)
}

\DeclareDocumentCommand{\lang}{m}{%
  \mathsf{L}(#1)
}

\DeclareDocumentCommand{\shorter}{O{}}{%
  \ll_{#1}
}
\DeclareDocumentCommand{\vshorter}{O{}}{%
  \rotatebox[origin=c]{90}{$\shorter[#1]$}
}

\DeclareDocumentCommand{\shorterdiff}{o}{%
  \IfNoValueTF{#1}{\ll}{\ll_{(#1)}}
}
\DeclareDocumentCommand{\vshorterdiff}{o}{%
  \rotatebox[origin=c]{90}{$\shorterdiff[#1]$}
}

\DeclareDocumentCommand{\pos}{m}{%
  \mathsf{pos}(#1)
}

\def\cursor#1{#1^{\triangleleft}}
\def\rpoint#1{#1_{\bullet}}
\def\pmax{\mathsf{pmax}}
\def\pmin{\mathsf{pmin}}
\def\sv{\mathsf{sv}}

\newcommand{\ltr}[1]{\mathtt{#1}}

\DeclareMathOperator{\GL}{GL}

\newcommand{\Q}{\mathbb{Q}}

\newcommand{\newextmathcommand}[2]{%
    \newcommand{#1}{\ensuremath{#2}\xspace}
}

\newextmathcommand{\A}{\mathcal A}
\newextmathcommand{\D}{\mathcal D}
\newextmathcommand{\mainbs}{\BS(1,q)}
\newextmathcommand{\PSPACE}{\mathsf{PSPACE}}

\newcommand{\BS}{\mathrm{BS}}
\newcommand{\Z}{\mathbb{Z}}
\newcommand{\N}{\mathbb{N}}
\newcommand{\calA}{\mathcal{A}}
\newcommand{\calB}{\mathcal{B}}
\newcommand{\calM}{\mathcal{M}}

\newtheorem{fact}[theorem]{Fact}

\title{Rational subsets of Baumslag-Solitar groups}

\author{Micha\"{e}l Cadilhac}
{DePaul University, Chicago, IL, USA}
{michael@cadilhac.name} 
{https://orcid.org/0000-0001-9828-9129} 
{} 

\author{Dmitry Chistikov}
{Centre for Discrete Mathematics and its Applications (DIMAP) \&\newline
 Department of Computer Science, University of Warwick, United Kingdom}
{d.chistikov@warwick.ac.uk} 
{https://orcid.org/0000-0001-9055-918X} 
{} 

\author{Georg Zetzsche}
{Max Planck Institute for Software Systems (MPI-SWS), Germany}
{georg@mpi-sws.org} 
{https://orcid.org/0000-0002-6421-4388} 
{} 

\authorrunning{M. Cadilhac, D. Chistikov, and G. Zetzsche}

\acknowledgements{The research for this work was carried out in part at the
  Autob\'{o}z Research Camp in 2019 in Firbush, Scotland.  The authors are very
  grateful to Piotr Hofman for comments that led to important insights and to
  Michael Blondin as well as to all other participants for discussions.  }



\Copyright{Micha\"{e}l Cadilhac, Dmitry Chistikov, and Georg Zetzsche}

\iflongversion
  \hideLIPIcs
\fi
\EventEditors{Artur Czumaj, Anuj Dawar, and Emanuela Merelli}
\EventNoEds{3}
\EventLongTitle{47th International Colloquium on Automata, Languages, and Programming (ICALP 2020)}
\EventShortTitle{ICALP 2020}
\EventAcronym{ICALP}
\EventYear{2020}
\EventDate{July 8--11, 2020}
\EventLocation{Saarbrücken, Germany (virtual conference)}
\EventLogo{}
\SeriesVolume{168}
\ArticleNo{116}

\begin{document}

\maketitle

\begin{abstract}
  We consider the rational subset membership problem for
  Baumslag-Solitar groups. These groups form a prominent class 
  in the area of algorithmic group theory, and they were recently
  identified as an obstacle for understanding the rational subsets of
  $\GL(2,\Q)$.
  
  We show that rational subset membership for Baumslag-Solitar groups $\BS(1,q)$
  with $q\ge 2$ is decidable and $\PSPACE$-complete. To this end, we introduce a
  word representation of the elements of $\BS(1,q)$: their pointed expansion
  (PE), an annotated \(q\)-ary expansion.  Seeing subsets of \(\BS(1,q)\) as word
  languages, this leads to a natural notion of PE-regular subsets of
  \(\BS(1, q)\): these are the subsets of \(\BS(1,q)\) whose sets of PE are regular
  languages.  Our proof shows that every rational subset of $\BS(1,q)$ is
  PE-regular.

  Since the class of PE-regular subsets of $\BS(1,q)$ is well-equipped
  with closure properties, we obtain further applications of these
  results. Our results imply that (i)~emptiness of Boolean
  combinations of rational subsets is decidable, (ii)~membership to
  each fixed rational subset of $\BS(1,q)$ is decidable in logarithmic
  space, and (iii)~it is decidable whether a given rational subset is
  recognizable. In particular, it is decidable whether a given
  finitely generated subgroup of $\BS(1,q)$ has finite index.
\end{abstract}

\iflongversion
  \makeatletter
  \newpage
  \pagenumbering{gobble}
  \vskip0.7\bigskipamount
  \noindent
  \rlap{\color{lipicsLineGray}\vrule\@width\textwidth\@height1\p@}%
  \hspace*{7mm}\fboxsep1.5mm\colorbox[rgb]{1,1,1}{\raisebox{-0.4ex}{%
    \large\selectfont\sffamily\bfseries Contents}}
  \setcounter{tocdepth}{2}
  \setlength{\cftbeforesecskip}{0em}
  \vskip3pt

  {\renewcommand{\familydefault}{\sfdefault}\sffamily
    \@starttoc{toc}}
  \makeatother
\fi

\newpage
\pagenumbering{arabic}
\setcounter{page}{1}

\section{Introduction}
\subparagraph*{Subsets of groups}
Regular languages are an extremely versatile tool in algorithmics on
sets of finite words. This is mainly due to two reasons. First, they
are robust in terms of representations and closure properties: They
can be described by finite automata, by recognizing morphisms, and
by monadic second-order logic and they are closed under Boolean and an
abundance of other operations. Second, many properties (such as emptiness) are easily
decidable using finite automata.

Given this success, there have been several attempts to develop an analogous
notion for subsets of (infinite, finitely generated) groups. Adapting the notion
of recognizing morphism yields \emph{recognizable subsets} of a group $G$. They
are closed under Boolean operations, and problems such as membership or
emptiness are decidable. However, since they are merely unions of cosets of
finite-index normal subgroups, their expressiveness is severely limited.

Another notion is that of \emph{rational subsets}, which transfer
(non-deterministic) finite automata to groups. Starting with
pioneering work by Benois~\cite{Benois1969} in 1969, they have matured
into an important tool in group theory. Rational subsets are quite
expressive: They include finitely generated submonoids and are closed
under (finite) union, pointwise product, and Kleene star. Moreover, they have
been applied successfully to solving
equations in groups~\cite{DIEKERT2005105,DBLP:conf/icalp/CiobanuE19},
as well as in
other settings~\cite{BaSi2010,silva2017automata}.

The high expressiveness of rational subsets comes at the cost of
undecidability of decision problems for many groups.  The most
fundamental one is the \emph{membership problem} for rational subsets:
Given a rational subset $R$ of a group $G$ and an element $g\in G$,
does $g$ belong to $R$?  Understanding for which groups this problem
is decidable received significant attention over the last two decades,
see~\cite{Lohrey2016survey} for a survey. Unfortunately, the rational
subsets do not quite reach the level of robustness of regular
languages. In general, the class of rational subsets of a group is not
closed under Boolean operations, and the
papers~\cite{DBLP:journals/ijac/LohreyS08,bazhenova2000rational} study
for which groups the rational subsets form a Boolean algebra.

\subparagraph*{Baumslag-Solitar groups} A prominent class of groups is
that of \emph{Baumslag-Solitar groups} $\BS(p,q)$. For each
$p,q\in\N$, the group is defined as
$\BS(p,q)=\langle a,t \mid ta^pt^{-1}=a^q\rangle$. They were
introduced in 1962 by Baumslag and Solitar to provide an example of a
two-generator one-relator group that is non-Hopfian. They recently
came into focus from the algorithmic perspective in a paper by
Kharlampovich, L\'opez, and
Miasnikov~\cite{KharlampovichLopezMiasnikov2019}, which shows that
solvability of equations is decidable in $\BS(1,q)$.  They have also
been studied from several other perspectives, such as the decidability
and complexity of the word
problem~\cite{Robinson1993,DBLP:conf/latin/DiekertMW14,Weiss2015}, the
conjugacy problem~\cite{DBLP:conf/latin/DiekertMW14,Weiss2015}, tiling problems~\cite{DBLP:journals/corr/AubrunK13}, and
computing normal
forms~\cite{DBLP:journals/ijac/DiekertL11,ELDER2013260,elder2010}.

More specifically to our setting, the Baumslag-Solitar groups have
recently been identified by Diekert, Potapov, and
Semukhin~\cite{DBLP:journals/corr/abs-1910-02302} as a stumbling block
in solving rational subset membership in the group $\GL(2,\Q)$, that is,
the group of invertible $2 \times 2$ matrices over $\Q$.
They show
that any subgroup of $\GL(2,\Q)$ containing $\GL(2,\Z)$ is either of the
form $\GL(2,\Z)\times\Z^k$ for $k\ge 1$ or contains $\BS(1,q)$ as a
subgroup for some $q\ge 2$.
Rational subset membership for $\GL(2,\Z)\times\Z^k$ is today a matter
of standard arguments~\cite{Lohrey2016survey},
because $\GL(2,\Z)$ is virtually free.
Therefore, making significant progress towards decidability in
larger subgroups requires understanding rational subsets of
$\BS(1,q)$. 

One can represent the elements of $\BS(1,q)$ as pairs $(r,m)$, where
$r$ is a number in $\Z[\tfrac{1}{q}]$,
say $r=\pm\sum_{i=-n}^n a_iq^i$
for $a_{-n},a_{-n+1},\ldots,a_n\in\{0,\ldots,q-1\}$,%
\footnote{%
    $\Z[\tfrac{1}{q}]$ denotes (the additive group of)
    the smallest subring of $(\Q, +, \cdot)$ containing $\Z$
    and $1/q$; as a set, it consists of all rational numbers
    of the form $n \cdot q^j$, $n, j \in \Z$.
  }
and
$m\in\Z$. Here, one can think of $m$ as a \emph{cursor} pointing to a
position in the $q$-ary expansion
$a_{n}q^{n}+\cdots+a_{-n}q^{-n}$. Then the action of the generators of $\BS(1,q)$
is as follows.
Multiplication by $t$ or $t^{-1}$ moves
the cursor to the left or the right, respectively. Multiplication by $a$
adds $q^m$; likewise, multiplication by $a^{-1}$ subtracts $q^m$. Thus,
from an automata-theoretic perspective, one can view the rational
subset membership problem as the reachability problem for an extended
version of one-counter automata. Instead of storing a natural number,
such an automaton stores a number $r \in \Z[\frac{1}{q}]$. Moreover, instead of
instructions ``increment by 1'' and ``decrement by 1'', it has an
additional $\Z$-counter $m$ that determines the value to be
added in the next update. Then, performing ``increment'' on $r$ will add $q^m$ and
``decrement'' on $r$ will subtract $q^m$. The $\Z$-counter $m$ supports
the classical ``increment'' and ``decrement'' instructions.

\subparagraph*{Contribution} Our \emph{first main contribution} is to show is
that for each group $\BS(1,q)$, the rational subset membership problem is
decidable and $\PSPACE$-complete. To this end, we show that each rational subset can be
represented by a regular language of finite words that encode elements of
$\BS(1,q)$ in the natural way: For $(r,m)$ as above, we encode each digit $a_i$
by a letter; and we decorate the digits at position $0$ and at position
$m$.  We call this encoding the \emph{pointed expansion} (PE) of \((r,m)\).
This leads to a natural notion of subsets of $\BS(1,q)$, which we call
\emph{PE-regular}. We regard the introduction of this notion as the \emph{second
  main contribution} of this work.

The class of PE-regular subsets of $\BS(1,q)$ has several properties that
make them a promising tool for decision procedures for $\BS(1,q)$:
First, our proof shows that it effectively includes the large class of
rational subsets, in particular any finitely generated submonoid.
Second, they form an effective Boolean algebra. Third, due to them
being regular languages of words, they inherit many algorithmic tools
from the setting of free monoids.
We apply these properties to obtain \emph{three applications of our main results}.
\begin{enumerate}
\item Membership in each fixed rational subset can be decided in
  logarithmic space.

\item We show that it is decidable whether a given PE-regular subset
  (and thus a given rational subset) is recognizable. Recognizability
  of rational subsets is rarely known to be decidable for groups: The
  only examples known to the authors are free groups, for which
  decidability was shown by
  S\'{e}nizergues~\cite{DBLP:journals/acta/Senizergues96} (and
  simplified by Silva~\cite{DBLP:journals/ita/Silva04}) and free
  abelian groups (this follows from \cite[Theorem
  3.1]{GinsburgSpanier1966a}). Since (i)~finitely generated subgroups
  are rational subsets and (ii)~a subgroup of any group $G$ is
  recognizable if and only if it has finite index in $G$, our result
  implies that it is decidable whether a given finitely generated
  subgroup of $\BS(1,q)$ has finite index. Studying decidability of
  this finite index problem in groups was recently proposed by
  Kapovich~\cite[Section~4.3]{DBLP:journals/dagstuhl-reports/DiekertKLM19}.
\item Our results imply that emptiness of Boolean combinations (hence
  inclusion, equality, etc.) of rational subsets is decidable. (We
  also show that the rational subsets of $\BS(1,q)$ are not closed
  under intersection.)  This is a strong decidability property that
  already fails for groups as simple as $F_2\times \Z$ (this follows
  from \cite[Theorem~6.3]{Ibarra1978}), where $F_2$ is the free group
  over two generators, and hence for $\GL(2,\Z)\times\Z^k$, $k\ge 1$.
\end{enumerate}
Finally, we remark that
since $\BS(1,q)$ is isomorphic to the group of all matrices
$\begin{psmallmatrix} q^m & r \\ 0 & 1\end{psmallmatrix}$ for $m\in\Z$
and $r\in\Z[\tfrac{1}{q}]$, our results can be interpreted as solving
the rational subset membership problem for this subgroup of $\GL(2,\Q)$.

\subparagraph{Related work} It is well-known that membership in a
given finitely generated subgroup, called the \emph{generalized word
  problem} of $\BS(1,q)$, is decidable. This is due to a general result of
Romanovski\u{\i}, who showed in \cite{Romanovskii1974a} and
\cite{Romanovskii1980a} that solvable groups of derived length two
have a decidable generalized word problem (it is an easy exercise to
show that $\BS(1,q)$ is solvable of derived length two for each
$q\in\N$).

Another restricted version of rational subset membership is the
\emph{knapsack problem}, which was introduced by Myasnikov, Nikolaev,
and Ushakov~\cite{MyNiUs14}. Here, one is given group elements
$g_1,\ldots,g_k,g$ and is asked whether there exist
$x_1,\ldots,x_k\in\N$ with $g_1^{x_1}\cdots g_k^{x_k}=g$. A recent
paper on the knapsack problem in Baumslag-Solitar groups by Dudkin and
Treyer~\cite{DudkinTreyer2018} left open whether the knapsack problem
is decidable in $\BS(1,q)$ for $q\ge 2$. This was settled very
recently in \cite{LohreyZetzsche2020a}, where one expresses
solvability of $g_1^{x_1}\cdots g_k^{x_k}=g$ in a variant of B\"{u}chi
arithmetic. A slight extension of that proof yields a regular
language as above for the set
$S=\{g_1^{x_1}\cdots g_k^{x_k}\mid x_1,\ldots,x_k\in\N\}$. Note that
each element $g_i$ moves the cursor either to the left (i.e.\
increases $m$), to the right (i.e.\ decreases $m$), or not at all.
Thus, in a product $g_1^{x_1}\cdots g_k^{x_k}$, the cursor direction
is reversed at most $k-1$ times. The challenge of our translation from
rational subsets to PE-regular subsets is to capture products where the
cursor changes direction an unbounded number of times.

Finally, closely related to rational subsets, there is another
approach to group-theoretic problems via automata: One can represent
finitely generated subgroups of free groups using \emph{Stallings
  graphs}
. Due to the special setting of free groups, they behave in many ways
similar to automata over words and are thus useful for decision
procedures~\cite{KAPOVICH2002608}. Stallings graphs have recently been
extended to semidirect products of free groups and free abelian groups
by Delgado~\cite{Delgado2017a}. However, this does not include
products $\Z[\tfrac{1}{q}]\rtimes\Z$ and is restricted to subgroups.




\section{Basic notions}\label{sec:prelim}
\subparagraph{Automata, rational subsets, and regular languages} Since
we work with automata over finite words and over groups, we define
automata over a general monoid $M$. A subset $S\subseteq M$ is
\emph{recognizable} if there is a finite monoid $F$ and a morphism
$\varphi\colon M\to F$ such that $S=\varphi^{-1}(\varphi(S))$. If $M$
is a group, one can equivalently require $F$ to be a finite group.

For a subset $S\subseteq M$, we
write $\langle S\rangle$ or $S^*$ for the submonoid \emph{generated by
  $S$}, i.e. the set of elements that can be written as a (possibly
empty) product of elements of $S$. In particular, the neutral element
$1\in M$ always belongs to $\langle S\rangle=S^*$. A \emph{generating
  set} is a subset $\Sigma\subseteq M$ such that
\(M = \langle\Sigma\rangle\). We say that $M$ is \emph{finitely
  generated (f.g.)}  if it has a finite generating set. Suppose $M$ is
finitely generated and fix a finite generating set $\Sigma$. An
\emph{automaton over $M$} is a tuple $\calA=(Q,\Sigma,E,q_0,q_f)$,
where $Q$ is a finite set of \emph{states},
$E\subseteq Q\times \Sigma\times Q$ is a finite set of \emph{edges},
$q_0\in Q$ is its \emph{initial state}, and $q_f\in Q$ is its
\emph{final state}. A \emph{run (in $\calA$)} is a sequence
$\rho=(p_0,a_1,p_1)\cdots (p_{m-1},a_m,p_m)$, where
$(p_{i-1},a_i,p_i)\in E$ for $i\in[1,m]$.  It is \emph{accepting} if
$p_0=q_0$ and $p_m=q_f$.  By $[\rho]$, we denote the \emph{production}
of \(\rho\), that is, the element $a_1\cdots a_m\in M$.  Two runs are
\emph{equivalent} if they start in the same state, end in the same
state, and have the same production.  For a set of runs $P$, we denote
$[P] = \{[\rho] \mid \rho \in P\}$.

The subset \emph{accepted by $\calA$} is
$\lang{\calA}=\{[\rho]\mid \text{$\rho$ is an accepting run in
  $\calA$}\}$. A subset $R\subseteq M$ is called \emph{rational} if it
is accepted by some automaton over $M$. It is a standard fact that the
family of rational subsets of $M$ does not depend on the chosen
generating set $\Sigma$. Rational subsets of a free monoid $\Gamma^*$
for some alphabet $\Gamma$ are also called \emph{regular languages}.
If
$M=\Gamma^*\times\Delta^*$ for alphabets
$\Gamma,\Delta$, then rational subsets of $M$ are also
called \emph{rational transductions}. If
$T\subseteq\Gamma^*\times\Delta^*$ and $L\subseteq\Gamma^*$,
then we set
$TL=\{v\in\Delta^* \mid \exists u\in L\colon (u,v)\in T\}$. It is
well-known that if $L\subseteq\Gamma^*$ is regular and
$T\subseteq\Gamma^*\times\Delta^*$ is rational, then $TL$ is
regular as well~\cite{Berstel1979}.

\subparagraph{Baumslag-Solitar groups} The \emph{Baumslag-Solitar
  groups} are the groups $\BS(p,q)$ for $p,q\in\N$, where
$\BS(p,q)=\langle a,t\mid ta^pt^{-1}=a^q\rangle$. They were introduced
in~1962 by Baumslag and Solitar~\cite{baumslag1962some} to provide an
example of a non-Hopfian group with two generators and one defining
relation. In this paper, we focus on the case $p=1$. In this case,
there is a well-known isomorphism
$\BS(1,q)\cong \Z[\tfrac{1}{q}]\rtimes\Z$ and we will identify the two
groups. Here, $\Z[\tfrac{1}{q}]$ is the additive group of number $nq^i$ with $n,i\in\Z$, and $\rtimes$ denotes
semidirect product. Building this semidirect product requires us to specify an
automorphism $\varphi_m$ of $\Z[\tfrac{1}{q}]$ for each $m\in\Z$,
which is given by $\varphi_m(nq^i)=q^m\cdot n q^i$.

For readers not familiar with semidirect products, we give an
alternative self-contained definition of $\Z[\tfrac{1}{q}]\rtimes\Z$.
The elements of this group
are pairs $(r,m)$, where $r\in\Z[\tfrac{1}{q}]$ and $m\in\Z$. The
multiplication is defined as
\[ (r,m)(r',m')=(r+q^{m}\cdot r',m+m'). \] We think of an element
$(r,m)$ as representing a number $r$ in $\Z[\tfrac{1}{q}]$ together
with a cursor $m$ to a position in the $q$-ary expansion of
$r$. Multiplying an element $(r,m)$ by the pair $(1,0)$ from the right
means adding $1$ at the position in \(r\) given by $m$, hence adding
$q^m$ to \(r\) and leaving the cursor unchanged: we have
$(r,m)(1,0)=(r+q^m,m)$. Multiplying by $(0,1)$ moves the cursor one
position to the left: $(r,m)(0,1)=(r,m+1)$.  It is easy to see that
$\Z[\tfrac{1}{q}]\rtimes\Z$ is generated by the set
$\{(1,0), (-1,0), (0,1), (0,-1)\}$.  The isomorphism
$\BS(1,q)\xrightarrow{\sim}\Z[\tfrac{1}{q}]\rtimes\Z$ mentioned above
maps $a$ to $(1,0)$ and $t$ to $(0,1)$. Since we identify
$\BS(1,q)$ and $\Z[\tfrac{1}{q}]\rtimes\Z$, we will have $a=(1,0)$ and $t=(0,1)$.  In particular, $a$ can be thought of as ``add''/``increment'', and $t$ as ``move''. We regard elements of the subgroup
$\Z[\tfrac{1}{q}] \times \{0\}$ of $\BS(1,q)$ as elements of
$\Z[\tfrac{1}{q}]$, i.e., integers or rational fractions with
denominator $q^i$, $i \ge 1$.

\subparagraph{Rational subset membership} Unless specified otherwise,
automata over $\BS(1,q)$ will use the generating set
$\Sigma=\{a,a^{-1},t,t^{-1}\}=\{(1,0), (-1,0), (0,1), (0,-1)\}$.  The
central decision problem of this work is the \emph{rational subset
  membership problem for $\BS(1,q)$}:
\begin{description}
\item[Given] An automaton $\calA$ over $\BS(1,q)$ and an element $g\in\BS(1,q)$ as a word over $\Sigma$.
\item[Question] Does $g$ belong to $\lang{\calA}$?
\end{description}

\subparagraph{Automata over BS(1,\,\emph{q})} In the following
definitions, let \(\calA = (Q,\Sigma,E,q_0,q_f)\) be an automaton over
\(\BS(1, q)\).
For a run $\rho$ of $\calA$, recall that $[\rho]\in \Z[\tfrac{1}{q}]\rtimes\Z$ is the
\emph{production} of $\rho$.  Moreover, if $[\rho]=(r,m)$ with
$r\in\Z[\tfrac{1}{q}]$ and $m\in\Z$, then we define $\pos{\rho}=m$, and call this the
\emph{final position} of \(\rho\).  More generally, the \emph{position} at a
particular point in~$\rho$ is the final position of the corresponding prefix
of~$\rho$.  By $\pmax(\rho)$, we denote the maximal value of $\pos{\pi}$ where
$\pi$ is a prefix of $\rho$.  Analogously, $\pmin(\rho)$ is the minimal value of
$\pos{\pi}$ where $\pi$ is a prefix of $\rho$.  A run $\rho$ is \emph{returning} if
$\pos{\rho}=0$.  It is \emph{returning-left} if in addition $\pmin(\rho)=0$.  Note that
for a returning run $\rho$, we have $[\rho]\in\Z[\tfrac{1}{q}]$ and if $\rho$ is
returning-left, we have $[\rho]\in\Z$.
Let $|\rho|$ be the length of the run $\rho$ as a word over $E$.
We will often write $\rho_i$ assuming $\rho = \rho_1 \rho_2 \dots \rho_\ell$ where
each $\rho_i \in E$ and $\ell = |\rho|$.
A run is a
\emph{cycle}   if it is
returning and starts and ends in the same state.
The \emph{thickness} of a run \(\rho\) is defined as the greatest number of times a
position is seen:
\[\thickness(\rho) = \max_{n \in \Z} |\{i \mid \pos{\rho_1\cdots\rho_i} = n\}|\enspace.\]
We call a run \emph{\(k\)-thin} if its thickness is at most \(k\).

We let \(\runs{\calA}\) (resp.\ \(\retruns{\calA}\), \(\leftruns{\calA}\)) be the set
of all accepting runs (resp.\ accepting returning runs, accepting returning-left
runs) of \(\calA\).  We add \(k\) in subscript to restrict the set to \(k\)-thin runs;
for instance, \(\retruns[k]{\calA}\) is the set of \(k\)-thin returning runs.
Further, we write \(\runs[k][p\to p']{\calA}\) for \(k\)-thin runs that start in
\(p\) and end in \(p'\), and use the similar notations
\(\retruns[k][p\to p']{\calA}\) and \(\leftruns[k][p\to p']{\calA}\).

Seeing \(\{0, \ldots, q-1\}\) as an alphabet, write \(\Phi_q\) for letters from this
alphabet with possibly a \(\bullet\) subscript (e.g., \(\rpoint{0}\)), a
\(\triangleleft\) superscript (e.g., \(\cursor{0}\)), or both (e.g.,
\(\cursor{\rpoint{0}}\)).  For \(v = (r, n) \in \BS(1, q)\), we write
\(\pe(v)\) for its base-\(q\) \emph{pointed expansion} (or just \emph{expansion}) as a word in
\(\pm \Phi_q^*\), where the subscript \(\bullet\) and the superscript
\(\triangleleft\) appear only once, the former representing the radix point, the latter
indicating the value of \(n\).  That is, if \(r = \sum_{i=-k_2}^{k_1} a_i q^i\), with
\(k_1, k_2 \geq 0\), \(\pe(v)\) is the following word:
\[\pm a_{k_1}\cdots a_1 (a_0)_\bullet a_{-1} \cdots a_{-k_2}\enspace,\]%
where \(\triangleleft\) is added to \(a_n\).
%
We tacitly assume a uniqueness condition: the
expansion \(\pe(v)\) of an element \(v \in \BS(1, q)\) is the shortest that
abides by the definition.
Expansions are read by automata in the left to right direction, i.e.,
from most to least significant digit.

\begin{definition}
\label{def:regular}
We say that a subset of \(R \subseteq \BS(1, q)\) is \emph{PE-regular}, where PE stands
for pointed expansion, if the word language \(\{\pe(v) \mid v \in R\}\) is regular.
\end{definition}

We remark that basic properties of regular languages support the
transformation of noncanonical expansions of elements \(\BS(1, q)\),
i.e., those with zeros on the left or right, into canonical ones,
\(\pe(v)\).  Finally, recall that we identify each
$r\in\Z[\tfrac{1}{q}]$ with
$(r,0)\in\Z[\tfrac{1}{q}]\rtimes\Z$. Hence, for $r\in\Z[\frac{1}{q}]$,
$\pe(r)$ is the $q$-ary expansion of $r$ (with $\triangleleft$ as an
additional decoration at the radix point).


\section{Main results}\label{results}
\newsavebox{\figone}
\newsavebox{\figtwo}

\tikzset{gadget/.style={->,>=stealth,initial text=,minimum size=7pt,auto,on grid,scale=1,inner sep=1pt,node distance=1cm}}
\tikzset{every state/.style={minimum size=15pt,inner sep=1pt,fill=black!10,draw=black!70,thick}}

\begin{lrbox}{\figone}
    \begin{tikzpicture}[gadget, node distance=2cm]
      \node[state,initial] (p1) {$p_1$};
      \node[state, right=of p1] (p2) {$p_2$};
      \node[state,accepting by arrow, right=of p2] (p3) {$p_3$};
      \path (p1) edge [loop above] node {$t^{-2}$} (p1);
      \path (p2) edge [loop above] node {$t^2$} (p2);
      \path (p2) edge [loop below] node {$tat$} (p2);
      \path (p3) edge [loop above] node {$t^{-2}$} (p3);
      \path (p1) edge node {$1$} (p2);
      \path (p2) edge node {$1$} (p3);
    \end{tikzpicture}
\end{lrbox}
\begin{lrbox}{\figtwo}
    \newcommand{\radius}{1.1cm}
    \begin{tikzpicture}[gadget, node distance=2cm]

      \node[state,initial,accepting by arrow] (p0) at (-90:\radius) {$p_0$};
      \node[state] (p1) at (-180:\radius) {$p_1$};
      \node[state] (p2) at (90:\radius) {$p_2$};
      \node[state] (p3) at (0:\radius) {$p_3$};

      \path (p0) edge [loop above] node {$t^{-1}$} (p0);
      \path (p0) edge node {$a$} (p1);
      \path (p1) edge node {$t$} (p2);
      \path (p2) edge node {$a$} (p3);
      \path (p3) edge node {$t$} (p0);

    \end{tikzpicture}
\end{lrbox}

\begin{figure}[t]
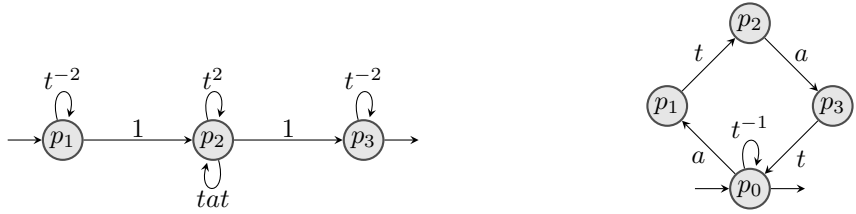


  \subcaptionbox{Automaton over $\BS(1,q)$ from \cref{non-closure-intersection}\label{automaton-intersection}.}[0.5\textwidth]{\usebox{\figone}}
  \subcaptionbox{Automaton over $\BS(1,2)$ from
    \Cref{ex:three}.\label{automaton-three}}[0.5\textwidth]{%
    \usebox{\figtwo}}
  \caption{Example automata over $\BS(1,q)$.}
\end{figure}

In this section, we list our main contributions, their proofs being deferred to
later sections.  Our first main result is that one can translate rational
subsets into PE-regular subsets.
\begin{theorem}\label{main-effective-regularity}
  Every rational subset of $\BS(1,q)$ is effectively PE-regular.
\end{theorem}
This will be shown in \cref{rational-to-regular}.
Since membership is decidable for regular languages and
given $g\in\BS(1,q)$ as a word over $\{a,a^{-1},t,t^{-1}\}$, one can
compute $\pe(g)$, \cref{main-effective-regularity} implies that
rational subset membership is decidable.  Our next main result is that
the problem is $\PSPACE$-complete.
\begin{theorem}\label{main-complexity}
  The rational subset membership problem for $\BS(1,q)$ is
  $\PSPACE$-complete.
\end{theorem}
This is shown in \cref{complexity}.
We shall also conclude that membership to each fixed rational subset
is decidable in logspace.
\begin{theorem}\label{main-complexity-fixed}
  For each fixed rational subset of $\BS(1,q)$, membership
  is decidable in logarithmic space.
\end{theorem}
The proof can also be found in \cref{complexity}.
Note that, in particular, membership to each fixed subgroup of
$\BS(1,q)$ is decidable in logarithmic space.
Another application of \cref{main-effective-regularity} is that one
can decide whether a given rational subset of $\BS(1,q)$ is
recognizable.
\begin{theorem}\label{main-recognizability}
  Given a PE-regular subset $R$ of $\BS(1,q)$, it is
  decidable whether $R$ is recognizable.
\end{theorem}
This is shown in \cref{recognizability}.
Since a subgroup of any group $H$ is recognizable
if and only if it has finite index in $H$ (see, e.g.~\cite[Prop.~3.2]{BaSi2010}), we obtain:
\begin{corollary}\label{main-finite-index}
  Given a f.g.\ subgroup of $\BS(1,q)$, it is
  decidable whether it has finite index.
\end{corollary}


\iflongversion
  \section{Closure properties}
  In this section, we show some closure properties of rational and PE-regular
  subsets of \(\BS(1, q)\).  Our goal is twofold: First, give a hands-on
  introduction to these concepts, and second, contrast them by exhibiting
  structural differences between these sets.

  \subsection{The PE-regular subsets of BS(1,\,\emph{q}) form a Boolean algebra}
\fi
\begin{restatable}{proposition}{mainClosureProperties}\label{main-closure-properties}
  The PE-regular subsets of $\BS(1,q)$ form an effective Boolean
  algebra. Moreover, for PE-regular subsets $R,S\subseteq\BS(1,q)$, the
  sets $RS=\{rs\mid r\in R,~s\in S\}$ and
  $R^{-1}=\{r^{-1}\mid r\in R\}$ are PE-regular as well.
\end{restatable}
\iflongversion
  \begin{proof}
    The first statement is due to the fact that the regular languages
    form an effective Boolean algebra and that the set of all $\pe(g)$
    for $g\in\BS(1,q)$ is regular.
    
    It is easy to construct an automaton $\calM$ over
    $\Gamma^*\times\Gamma^*\times\Gamma^*$, for suitable $\Gamma$, that
    accepts the relation
    $T=\{(\pe(g),\pe(h), \pe(gh)) \mid g,h\in\BS(1,q)\}$: It makes sure
    that the radix point of the word in the second component is aligned
    with the cursor position of the word in the first component. Then,
    multiplying the two elements amounts to adding up the $q$-ary
    expansions (see also \cref{lem:sum} for a more general
    statement). Given automata for $\pe(R)$ and $\pe(S)$, we can easily
    modify $\calM$ so as to accept
    $\{(\pe(g), \pe(h), \pe(gh)) \mid g\in R,~h\in S\}$.  Projecting to
    the third component then yields an automaton for the language
    $\pe(RS)$. A similar modification of $\calM$ leads to
    $\{(\pe(g),\pe(h),\pe(gh)) \mid g\in
    R,~h\in\BS(1,q),~\pe(gh)=\pe(1)\}$. Projecting to the second
    component yields an automaton for $\pe(R^{-1})$.
  \end{proof}
\fi
Together with \cref{main-effective-regularity}, this implies that emptiness of
Boolean combinations (hence inclusion, equality) is decidable for rational
subsets. To further highlight the advantages of PE-regular subsets, we also show
that the rational subsets of $\BS(1,q)$ are not closed under intersection.

\iflongversion
  \subsection{The rational subsets of BS(1,\,\emph{q}) are not closed under
    intersection}\label{non-closure-intersection}

  We present an example of rational subsets $R_1,R_2\subseteq\BS(1,q)$
  such that $R_1\cap R_2$ is not rational. Let $R$ be the rational
  subset accepted by the automaton in \cref{automaton-intersection}. In
  $p_1$, it moves the cursor an even number of positions to the
  right. In $p_2$, it moves an even number of positions to the left and
  on the way, it adds $q$ in a subset of the even positions. In $p_3$,
  it moves to the right again. Then $R$ contains all elements
  $(r,m)\in\Z[\tfrac{1}{q}]\rtimes\Z$ where $r=\sum_{i\in A} q^{2i+1}$
  for some finite $A\subseteq\Z$ and $m\in 2\Z$. Now consider the sets
  $R_1=aR$, $R_2=Ra$, and their intersection $I=R_1\cap R_2$. Then we
  have $(r,m)\in R_1$ if and only if $r=1+\sum_{i\in A} q^{2i+1}$ and
  $m\in 2\Z$ for some finite $A\subseteq\Z$. Moreover, $(r,m)\in R_2$ if
  and only if $r=q^m+\sum_{i\in A} q^{2i+1}$ and $m\in 2\Z$ for some
  finite $A\subseteq\Z$.  Therefore, we have $(r,m)\in R_1\cap R_2$ if
  and only if $r=1+\sum_{i\in A} q^{2i+1}$ and $m=0$ for some finite
  $A\subseteq\Z$. Using the following \lcnamecref{bounded-precision}, we
  shall conclude that $I=R_1\cap R_2$ is not rational.
  \begin{restatable}{lemma}{boundedPrecision}\label{bounded-precision}
    Let $R\subseteq\Z[\tfrac{1}{q}]\rtimes\Z$ be a rational subset. If
    $R\subseteq\Z[\tfrac{1}{q}]\times\{0\}$, then there is a $k\in\N$
    with $R\subseteq\tfrac{1}{q^k}\Z\times\{0\}$.
  \end{restatable}
  Intuitively, this says that if all elements in a rational subset have
  the cursor in the origin, then its elements must have bounded
  precision. This can be shown using a pumping argument: If $R$ did contain
  elements with high powers of $q$ in the denominator, then the cursor
  must move arbitrarily far to the right, but then it can also end up to
  the right of the origin, which is impossible. Since
  $I\subseteq\Z[\tfrac{1}{q}]\times\{0\}$ contains $(1+q^{-2i+1},0)$ for
  any $i\in\N$, it cannot be rational.

  For the detailed proof of \cref{bounded-precision}, it is more
  convenient to argue with the well-known observation that an automaton
  that accepts a fixed element has to encode the element read so far in
  its state.  Let us make this formal. If $\calA=(Q,\Sigma,E,q_0,F)$ is
  an automaton over a group $G$, then a \emph{state evaluation} is a map
  $\eta\colon Q\to G$ such that $\eta(q_0)=1$ and for every edge
  $(p,g,p')\in E$, we have $\eta(p')=\eta(p)g$. Hence, a state
  evaluation assigns to each state $p$ a fixed group element $\eta(p)$
  such that on any path from $q_0$ to $p$, $\calA$ reads $\eta(p)$. An
  automaton is called \emph{trim} if (i)~every state is reachable from
  an initial state and (ii)~from every state, one can reach a final
  state.

  \begin{lemma}\label{state-evaluation}
    Let $\calA$ be a trim automaton over a group $G$ that accepts the
    set $\{1\}$. Then $\calA$ admits a state evaluation.
  \end{lemma}
  \begin{proof}
    Since $\calA$ is trim, we can choose $\eta\colon Q\to G$ such that
    for every $p\in Q$, there is a run from $q_0$ to $p$ in $\calA$ that
    reads $\eta(p)$.

    The fact that $\calA$ accepts $\{1\}$ implies that there is only one
    such $\eta$: Suppose $\rho_1$, $\rho_2$ are runs from $q_0$ to $p$
    and $\rho$ is a run from $p$ to a final state. Then since $\calA$
    accepts $\{1\}$, we have $[\rho_1][\rho]=1=[\rho_2][\rho]$ and thus
    $[\rho_1]=[\rho_2]$. Hence, $\eta$ is uniquely determined.
    
    This implies that $\eta$ is a state evaluation: We must have
    $\eta(q_0)=1$, because of uniqueness of $\eta$. Moreover, if there
    is an edge $(p,g,p')$, then we can pick a run $\rho$ from $q_0$ to
    $p$ and by uniqueness of $\eta$, we have
    $\eta(p')=[\rho]g=\eta(p)g$.
  \end{proof}

  Using \cref{state-evaluation}, we are ready to prove \cref{bounded-precision}.
  \begin{proof}
    Suppose $\calA$ is an automaton over $\Z[\tfrac{1}{q}]\rtimes\Z$
    that accepts a subset of $\Z[\tfrac{1}{q}]\times\{0\}$. Without loss
    of generality, we may assume that $\calA$ is trim and every edge has
    a label in $\{t,t^{-1},a,a^{-1}\}$. Consider the automaton $\calA'$
    obtained from $\calA$ by projecting to the right component. Then
    $\calA'$ is a trim automaton over $\Z$ that accepts
    $\{0\}$. According to \cref{state-evaluation}, $\calA'$ admits a
    state evaluation $\eta\colon Q\to\Z$. Since $Q$ is finite, the image
    of $\eta$ is included in some interval $[-k,k]$.

    This implies that for any state $p$ of $\calA$, any element $(r,m)$
    read on a path from $q_0$ to $p$ satisfies $m\in[-k,k]$. Therefore,
    every edge labeled $a^{\pm 1}$ adds a number $s=\pm q^m$ with $m\in[-k,k]$
    to the left component. Since in this case
    $s\in \tfrac{1}{q^k}\Z$, the \lcnamecref{bounded-precision}
    follows.
  \end{proof}

  \subsection{The PE-regular subsets of BS(1,\,\emph{q}) are not closed under
    iteration}\label{non-closure-iteration}

  The subset $A=\{(1+2^{-i},0) \mid i\ge 1\}$ of $\BS(1,2)$ is PE-regular, because
  $\pe(A)=\cursor{\rpoint{1}}0^*1$ is a regular language.  Let us now prove that
  the set $A^*$ is indeed not PE-regular. We begin with an auxiliary lemma.
  \begin{lemma}\label{summing-up}
    Suppose $k,m\ge 0$ and $1\le d_1\le\cdots\le d_k$ and $1\le e_1<e_2<\cdots<e_\ell$ with
    \begin{equation} \sum_{i=1}^k (1+2^{-d_i})=m+\sum_{i=1}^\ell 2^{-e_i} \label{two-sums}\end{equation}
    Then $m\ge \ell$.
  \end{lemma}
  \begin{proof}
    We prove $m\ge k$ and $k\ge\ell$.  We begin with $m\ge k$. Let $s$
    be the value of the two sums. Then clearly $k\le s$ and $s<m+1$,
    hence $k\le m+1$. Since both $k$ and $m$ are integers, it is
    impossible that $k>m$. Thus $k\le m$.

    The inequality $k\ge\ell$ follows by induction on $k$. Suppose that
    \cref{two-sums} holds and we add $1+2^{-d_{k+1}}$. We distinguish two cases:
    \begin{itemize}
    \item If in the binary expansion on the right, there is no
      digit $2^{-d_{k+1}}$, then the new binary expansion gains one
      $1$ digit and hence $\ell$ increases by one.
    \item If there already is a digit at $2^{-d_{k+1}}$, then the new
      binary expansion is obtained by flipping some $r\ge 1$ digits
      from $1$ to $0$ and flipping one $0$ into a $1$. Hence, $\ell$
      drops by $r$ and rises by $\le 1$.
    \end{itemize}
    In any case, the value for $\ell$ rises by at most one. This proves
    $k\ge \ell$.
  \end{proof}


  We regard $\Z[\tfrac{1}{q}]$ as a subset of
  $\Z[\tfrac{1}{q}]\rtimes\Z$ by identifying $r\in\Z[\tfrac{1}{q}]$ with
  $(r,0)\in\Z[\tfrac{1}{q}]\rtimes \Z$. Then in particular for $m\in\Z$,
  $\pe(m)\in\pm\{0,\ldots,q-1\}^*\{\cursor{\rpoint{0}},\ldots,\cursor{\rpoint{(q-1)}}\}$
  is the $q$-ary expansion of $m$, with the additional $\cursor{{}}$
  and $\rpoint{{}}$ at the right-most digit.
  \begin{lemma}\label{smallest-integer}
    Let $n\in\N$. Then $n$ is the smallest number $m\in\N$ with
    $\pe(m)\cdot 1^n\in \pe(A^*)$.
  \end{lemma}
  \begin{proof}
    Since $n+2^{-1}+\cdots+2^{-n}=\sum_{i=1}^n (1+2^{-i})$ clearly
    belongs to $A^*$, we have $\pe(n)\cdot 1^n\in\pe(A^*)$. Now
    suppose $\pe(m)\cdot 1^n\in\pe(A^*)$. Then we have
    \[ m+2^{-1}+\cdots 2^{-n}=\sum_{i=1}^k (1+2^{-d_i}) \]
    for some $k\ge 0$ and some $1\le d_1\le d_2\le \cdots \le d_k$. By
    \cref{summing-up}, this implies $m\ge n$.
  \end{proof}

  Now \Cref{smallest-integer} allows us to show that $\pe(A^*)$ is not
  regular.  Recall that for a language $L\subseteq\Gamma^*$, a
  \emph{right quotient} is a set of the form
  $Lu^{-1}:=\{v\in\Gamma^* \mid vu\in L\}$. Since a regular language has
  finite syntactic monoids (see, e.g.~\cite{Berstel1979}), it has only
  finitely many right quotients.
  Suppose $\pe(A^*)$ is regular. For each $n\in\N$, consider
  the right quotient $Q_n=\pe(A^*)(1^n)^{-1}$.
  Then according to \cref{smallest-integer}, for each $n\in\N$,
  $n$ is the smallest number $m$ with $\pe(m)\in Q_n\cap
  \pe(\Z)$. Thus, the sets
  $Q_0,Q_1,Q_2,\ldots$ are pairwise distinct, contradicting the fact
  that $\pe(A^*)$ has only finitely many right quotients.
\fi


\section{Every rational subset of BS(1,\,\emph{q}) is effectively PE-regular}\label{rational-to-regular}


In this section, we prove \cref{main-effective-regularity}.
We first illustrate our approach on an example.

\begin{example}
\label{ex:three}
Consider the automaton over $\BS(1,2)$ in \Cref{automaton-three}.
In its only initial and final state~$p_0$, it has a choice of two operations:
(i) move the cursor one position to the right (i.e. multiplication by $t^{-1}$) or
(ii) perform the increment on two neighbouring cells
     and stop one position left of them (i.e. multiplication by $atat$).
The automaton can perform these operations arbitrarily many times in any order.

We shall prove that the automaton accepts
\[ R=\{( 3 n \cdot 2^{m-2k}, m) \mid n \in \N,\ k\in \N,\ m \in \Z,\ 0\ge m-2k,\ 3n\cdot 2^{m-2k} \ge f(m,k) \}\enspace, \]
where
\[ f(m,k)=\sum_{i=1}^k 3\cdot 2^{m-2i}=\negthickspace\sum_{j=m-2k}^{m-1}\negthickspace 2^j = 2^m - 2^{m-2k} \enspace. \]

The language $\pe(R)$ is regular. Indeed, note that the number
$f(m,k)$ has a particularly simple binary representation. A pointed
expansion of $(r,m)$ belongs to $\pe(R)$ if there is a position
$m-2k\le 0$ such that reading the digits left of position $m-2k$
yields a number (namely $3n$) that (a)~is divisible by $3$ and
(b)~lies above a bound with a simple binary expansion.

Let us now prove that the automaton accepts $R$. Let $\rho$ be an
accepting run producing $(r,m)$. Choose $k\in\N$ so that
$\pmin(\rho)=m-2k$ or $\pmin(\rho)=m-2k+1$ (depending on whether
$m-\pmin(\rho)$ is even or odd).  Then $0\ge\pmin(\rho)\ge m-2k$.
Each time operation~(ii) is performed from position $\ell \in \Z$, the
update is $(r, m) \to (r + 3 \cdot 2^\ell, m+2)$.

Now, once $\rho$ visits position~$\pmin(\rho)$, in order to eventually
reach a position $\ell>\pmin(\rho)$, the operation~(ii) must be
performed on some position $\ge \ell-2$. In particular, to reach
position $m$, it must be performed at some position $m_1\ge m-2$.  If
$m_1>\pmin(\rho)$, to reach $m_1$, it must also be performed at some position
$m_2\ge m-4$, etc. Therefore, $\rho$ has to perform~(ii) at positions
$m_i\ge m-2i$ for each $i$ with $m>m-2i\ge \pmin(\rho)-1$.  In other
words, it has to do this for each $i=1,\ldots,k$.  Each time $\rho$
performs~(ii) at $m_i$, it adds $3\cdot 2^{m_i}$. Moreover, each
extra time $\rho$ performs~(ii), it adds a multiple of
$3\cdot 2^{m-2k}$, because $\pmin(\rho)\ge m-2k$. Thus, the number
produced in total is some $3n\cdot 2^{m-2k}$ where
\[ 3n\cdot 2^{m-2k}\ge \sum_{i=1}^k 3\cdot 2^{m_i}\ge \sum_{i=1}^k 3\cdot 2^{m-2i}=f(m,k) \enspace.\]


Conversely, suppose
$n \in \N$ and $k\in\N$, $m \in \Z$, $0\ge m-2k$, and $3n\cdot 2^{m-2k}\ge f(m,k)$.
The automaton first moves to position~$m-2k$ using operation (i). Then, it performs operations~(ii), (i), and (i) again, $\ell$ times in a loop (we specify $\ell$ later). That way, it adds $3\ell\cdot 2^{m-2k}$. Then, it moves to position $m$ by applying operation~(ii) exactly $k$ times. Hence, it applies~(ii) at positions $m-2i$ for $i=1,\ldots,k$ and each time, it adds $3\cdot 2^{m-2i}$. In total, the effect is
\[ 3\ell\cdot 2^{m-2k} + \sum_{i=1}^k 3\cdot 2^{m-2i} = 3\ell\cdot 2^{m-2k} + f(m,k) \enspace.\]
Since $3n\cdot 2^{m-2k}\ge f(m,k)$ and $f(m,k)$ is an integer multiple of $3\cdot 2^{m-2k}$, we can choose $\ell \in \N$ so as to produce $3n\cdot 2^{m-2k}$.%
\hfill\(\vartriangleleft\)
\end{example}


Following this example, we first show that any run has the same production as a
thin (i.e. bounded thickness) run in which thin returning-left cycles are
inserted (p.~\pageref{par:decomp}); in the example, such a cycle applies
operations~(ii), (i), and (i).  We then prove that the productions of thin runs
form a PE-regular set (p.~\pageref{par:thinreg}); in the example, the thin run
moves to the right to position $\pmin(\rho)$ using operation~(i) and then left
to $m \ge \pmin(\rho)$ using operations~(i) and~(ii).  Finally, we show that iterating returning-left
thin cycles also leads to a PE-regular set (p.~\pageref{par:itcyc}); in the
example, this is how we get all numbers divisible by~$3$ above a particular bound.  We combine these
three statements to prove \Cref{main-effective-regularity}.

In combining the thin run with cycles, we will need to ensure that the cycles
are anchored on the correct state.  To this end, we introduce an
annotated version of \(\pe([\rho])\) as follows.
Let $\calA$ be an automaton over $\BS(1,q)$ with state set \(Q\).  Let \(\rho\) be a
run in \(\calA\) starting and ending in arbitrary states and with
\([\rho] = (r, m)\).  Letting \(\bar{Q} = \{\bar{p} \mid p \in Q\}\) be a copy of
\(Q\), we define \(\sv(\rho)\), the \emph{state view} of \(\rho\), to be the word over the
alphabet \(\Phi_q \cup Q \cup \bar{Q} \cup \{\pm\}\) built as follows.  First, write:
\(\pe([\rho]) = \pm a_{k_1}\cdots a_1a_0a_{-1}\cdots a_{-k_2},\)%
where \(a_0\) has subscript \(\bullet\).  Second, let \(P_i \in (Q \cup \bar{Q})^{|Q|}\), for
\(i \in \{-k_2, \ldots, k_1\}\), be a word that contains all the states of \(Q\) once in a
fixed ordering of \(Q\),
either with a bar or not; the states without a bar are exactly those that
visit position \(i\) in \(\rho\).  That is,
\(p\) appears in \(P_i\) iff there is a prefix of \(\rho\) ending in \(p\) whose final
position is \(i\).  The state view of \(\rho\) is then:%
\[\sv(\rho) = \pm a_{k_1} \cdot P_{k_1}\cdots a_0 \cdot P_0 \cdot a_{-1} \cdot P_{-1} \cdots
a_{-k_2} \cdot P_{-k_2}\enspace.\]%
We naturally extend \(\sv\) to sets of runs.

\def\secorpar#1{\iflongversion\subsection{#1}\fi}

\secorpar{Any run is equivalent to a thin run augmented with thin
  returning-left cycles}\label{par:decomp}

We now focus on two properties of runs: the states they visit in the automaton
and the final position of their prefixes.  To that end, we introduce the
following notions.  For \(Q\) a finite set, a \emph{position path} is a word
\(\pi \in (Q \times \Z)^*\).
We extend the analogy with graphs calling elements of \(Q \times \Z\) \emph{vertices},
talking of the vertices \emph{visited} by a position path, and using the notion
of (position) \emph{subpaths} and \emph{cycles}.  The \emph{thickness} of a
position path \(\pi\) is defined as:
\[\thickness(\pi) = \max_{n \in \Z} |\{i \mid \pi_i = (q, n)\text{ for some }q\}|\enspace.\]

\begin{restatable}{lemma}{positionpaths}\label{lem:graph}
  Let \(Q\) be a finite set and \(\pi \in (Q \times \Z)^*\) be a position path.  For any
  subset \(V'\) of the vertices visited by \(\pi\), there exists a subpath
  \(\pi'\) of \(\pi\) such that:
  \begin{enumerate}
  \item \(\pi'\) starts and ends with the same vertices as \(\pi\),
  \item \(\pi'\) visits all the vertices in \(V'\),
  \item \(\thickness(\pi') \leq |Q|\cdot(1+2|V'|)\),
  \item \(\pi - \pi'\) consists only of cycles.
  \end{enumerate}
\end{restatable}
\iflongversion
  \begin{proof}
    We consider the directed multigraph \(G\) that is described by \(\pi\): the vertices
    in \(G\) are those appearing in \(\pi\), and an edge appears in \(G\) as many times as
    it does in \(\pi\).  Note that in \(G\), the in- and out-degrees of any vertex are
    equal, but for the start and end vertices of \(\pi\).

    We first note that Point 4 is true of any subpath \(\pi'\) that satisfies Point~1.
    Indeed, removing \(\pi'\) from \(G\) turns all the vertices into vertices with same
    in- and out-degrees.

    We build \(\pi'\) iteratively.  We first let \(\pi'\) be a shortest path from the
    starting vertex of \(\pi\) to its final vertex in \(G\); since it does not repeat
    any node in \(V\), its thickness is bounded by \(|Q|\).

    Now if \(\pi'\) visits all the vertices in \(V'\), we are done.  Otherwise, let \(v\)
    be a vertex in \(V'\) that \(\pi'\) does not visit; we augment \(\pi'\) with a cycle
    that includes \(v\) as follows.  Consider any shortest path from the start vertex of $\pi$ to \(v\) in
    \(G\), and let \(u\) be the last vertex of that path that appears in \(\pi'\).  Write
    \(\rho\) for the path from \(u\) to \(v\).  Since
    \(\pi - \pi'\) is a union of cycles, there is a path \(\rho'\) from \(v\) to \(u\) in \(\pi -
    \pi'\) (more details follow).  We can thus augment \(\pi'\) with the path \(\rho\rho'\) rooted at \(u\), potentially
    increasing the thickness of \(\pi'\) by \(2|Q|\).

    (In more detail, to find the path \(\rho'\), we argue as follows.
    The set of edges of \(\pi - \pi'\) forms an Eulerian multigraph,
    and so in \(\pi - \pi' - \rho\) the difference between outdegree
    and indegree is \(1\) for \(v\), \(-1\) for \(v\), and \(0\) for all
    other vertices. Therefore, constructing a walk edge by edge, starting
    from \(v\), while possible, will necessarily lead to a dead end
    at the vertex \(u\). Removing cycles from this walk will give
    a path \(\rho'\) from \(v\) to \(u\), as required.)
  \end{proof}
\fi

\begin{corollary}\label{cor:decomp}
  Let \(\calA\) be an automaton over \(\BS(1, q)\) with state set \(Q\), and let
  \(k = |Q|+2|Q|^2\).  Any run of \(\calA\) is equivalent to a run in
  \(\runs[k]{\calA}\) on which, for each state \(p\) appearing in the run, cycles
  from \(\leftruns[k][p\to p]{\calA}\) are inserted at an occurrence of \(p\) with
  smallest position.

  Conversely, any run built by taking a run in \(\runs[k]{\calA}\) and inserting
  cycles from \(\leftruns[k][p\to p]{\calA}\) at an occurrence of \(p\) is a run of
  \(\calA\).
\end{corollary}
\begin{proof}
  The converse is clear, 
  we thus focus on the first direction.

  \emph{(Step 1: Decomposing a run into a thin run and cycles.)}\quad Let
  \(\rho \in \runs{\calA}\), and extract from it a position path \(\pi = \pi_0\cdots\pi_{|\rho|}\) as follows.  We
  let, \(\pi_0 = (q_0, 0)\) and for all \(i \geq 1\):
  \[\pi_i = (p, n) \text{ where } \rho_i = (\cdot, \cdot, p) \text{ and } n = \pos{\rho_1\cdots
    \rho_i}\enspace.\]%
  For each state \(p\) visited by \(\rho\), let
  \(n_p= \min \{n \mid \text{there exists $i$ such that $\pi_i = (p, n)$}\}\); in
  words, \(n_p\) is the smallest final position of a prefix of \(\rho\) ending in
  \(p\).  Using \(V' = \{(p, n_p) \mid \rho \text{ visits } p\}\), 
  \cref{lem:graph} provides a position path \(\pi'\) of thickness
  \(\leq k = |Q|+2|Q|^2\) visiting all of \(V'\).

  From \(\pi'\), we can obtain the corresponding subpath \(\rho'\) of \(\rho\) that has the same
  starting and ending state and positions as \(\rho\), and such that \(\rho\) is made of
  \(\rho'\) onto which cycles are added.  The thickness of \(\rho'\) is bounded by \(k\), but
  the cycles can be of any thickness.

  \emph{(Step 2: Thinning the cycles.)}\quad Consider a cycle \(\beta\) that gets added to
  \(\rho'\) to form \(\rho\), say at position \(i\)
  (after initial \(i\) moves, \(\rho'_1 \cdots \rho'_i\)),
  and assume that
  \(\thickness(\beta) > k\).  Since a position is repeated more than \(k > |Q|\) times,
  there is a cycle \(\beta'\) \emph{within} \(\beta\) with
  \(\thickness(\beta') \leq k\); write then
  \(\beta = \alpha\cdot\beta'\cdot\alpha'\).  Let \(p\) be the state in
  \(\beta'\) that has the smallest position, that is, \(p\) is the ending state of
  the prefix \(\gamma\) of \(\beta'\) with final position \(\pmin(\beta')\); write
  \(\beta' = \gamma\cdot\gamma'\).  By definition, we have
  \(\pos{\rho'_1\cdots\rho'_i\alpha\gamma} \geq n_p\).  Note that
  \(\gamma'\cdot\gamma\) is in \(\leftruns[k][p\to p]{\calA}\).  We now remove
  \(\beta'\) from \(\beta\) and then insert \(\gamma'\cdot\gamma\) at the position
  \(j\) in \(\rho'\) that is such that \(\rho'_1\cdots\rho'_j\) ends in \(p\) with final position
  \(n_p\).  For the contribution of \(\gamma'\cdot\gamma\) to be the same as that of
  \(\beta'\) in the original path, we insert it
  \(q^d\) times, where
  \( d = \pos{\rho'_1\cdots\rho'_i\alpha\gamma} - n_p \).

  This shows that if any cycle added to \(\rho'\) is of thickness \(> k\), then a
  subcycle of it can be moved to another position of \(\rho'\) as a returning-left
  cycle.  Iterating this process, all the cycles added to \(\rho'\) will thus be of
  thickness \(\leq k\).  Moreover, if an added cycle \(\beta\) is not returning-left after
  these operations,
  or if it does not sit at an occurrence of its initial state with \emph{smallest} position,
  this means that we can decompose it just as above as
  \(\gamma\cdot\gamma'\), with \(\gamma\) reaching \(\pmin(\beta)\), and move
  \(\gamma'\cdot\gamma\), a returning-left cycle, to an appropriate position in \(\rho'\) as before.
\end{proof}

\secorpar{Intermezzo: reflecting on \Cref{cor:decomp}}\label{par:intz}

Before we continue with the
proof, we want to illustrate how crucial the previous
\lcnamecref{cor:decomp} is. \Cref{lem:graph} tells us that we can
obtain every run from a thin run by then adding cycles.  This already
simplifies the structure of $\runs{\calA}$: indeed, inserting cycles
at a certain position in a run~$\rho \in \runs{\calA}$ corresponds (in
algebraic terms) to adding to $[\rho]$ a subset of $\Z[\tfrac{1}{q}]$
closed under addition, i.e., a submonoid.  (Closure under addition
follows from the observation that any two returning cycles from each
$\retruns[k][p\to p]{\calA}$ can be concatenated.)

Sometimes one can conclude that every submonoid of a monoid has a
simple structure.  For example, every submonoid $M$ of $\Z$ is
semilinear and hence a PE-regular subset of $\Z[\frac{1}{q}]$.
Unfortunately, the situation in $\Z[\tfrac{1}{q}]$ is not as simple as
in $\Z$:
\iflongversion
  \begin{fact}\label{uncountable-submonoids}
    The group $\Z[\tfrac{1}{q}]$ has uncountably many submonoids.
  \end{fact}
  \begin{proof}
    Let $q\ge 2$. Consider the functions $f\colon\N\to\Z$ that satisfy $f(0)=0$ and
    \[ q\cdot f(i)-1 \le f(i+1)\le q\cdot f(i) \]
    for every $i\ge 1$. Note
    that there are uncountably many such functions $f$: One can
    successively choose $f(1), f(2), f(3),\ldots$ and has two options
    for each value. Consider the set
    \[ M_f=\left\{\left.\frac{n}{q^i} ~\right|~ n\ge f(i)\right\}.\]
    We claim that for any
    $n,i\in\N$, we have $\tfrac{n}{q^i}\in M_f$ iff $n\ge f(i)$.  (In
    other words, it cannot happen that $\frac{n}{q^i}$ can be
    represented as $\frac{m}{q^j}$ such that $n\ge f(i)$ but not
    $m\ge f(j)$.) For this, we have to show that $n\ge f(i)$ if and only
    if $qn\ge f(i+1)$. But if $n\ge f(i)$, then
    $qn\ge q\cdot f(i)\ge f(i+1)$ by choice of $f$. Conversely, if
    $qn\ge f(i+1)$, then $n\ge \tfrac{1}{q}f(i+1)\ge f(i)-\tfrac{1}{q}$,
    which implies $n\ge f(i)$ because $n$ and $f(i)$ are integers.  This
    proves the claim.

    The claim implies that $M_f$ is a submonoid of
    $\Z[\tfrac{1}{q}]$: For $\frac{n}{q^i},\frac{m}{q^j}\in M_f$ with
    $i\le j$, we have
    $\tfrac{n}{q^i}+\frac{m}{q^j}=\frac{q^{j-i}n+m}{q^j}$ and since
    $m\ge f(j)$, we clearly also have $q^{j-i}n+m\ge f(j)$ and thus
    $\tfrac{n}{q^i}+\frac{m}{q^j}\in M_f$. Moreover, since $f(0)=0$, we
    have $0=\tfrac{0}{q^0}\in M_f$.

    Finally, the claim implies that the mapping $f\mapsto M_f$ is injective: Determining
    $f(i)$ amounts to finding the smallest $n\in\N$ with $\tfrac{n}{q^i}\in M_f$.
  \end{proof}
\fi
Thus, $\Z[\tfrac{1}{q}]$ has submonoids with undecidable membership problem;
moreover, there is no hope for a finite description for every submonoid as in
$\Z$.  Thus, we need to look at our specific submonoids.  A simple observation
similar to \cref{lem:graph} allows us to obtain every run from a thin part by
adding \emph{thin} cycles. Hence, the submonoids that we add are of the form
$[\retruns[k][p\to p]{\calA}]^*$.  It is not hard to show (see \cref{lem:thinreg})
that $[\retruns[k][p\to p]{\calA}]$ is always a PE-regular set. Thus, one may hope
to prove that the regularity of $[\retruns[k][p\to p]{\calA}]$ implies regularity
of $[\retruns[k][p\to p]{\calA}]^*$. (This was an approach to rational subset
membership proposed by the third author of this work in \cite[Section
4.7]{DBLP:journals/dagstuhl-reports/DiekertKLM19}.) However,
\cref{non-closure-iteration} tells us that even for PE-regular
$R\subseteq\BS(1,q)$, the set $R^*$ may not be PE-regular.

Therefore, \Cref{cor:decomp} is the key insight of our proof. It says that a run
can be decomposed into a thin part and thin \emph{returning-left} cycles. Since
returning-left cycles produce integers, this will lead us to submonoids of $\Z$.

\secorpar{Sets of thin runs are PE-regular}\label{par:thinreg}

For the proof of that statement, we rely on the following result.  It is a
classical exercise to show that automata can compute the addition of numbers in
a given base. We rely on a slight extension: Using the base-\(q\)
\emph{signed-digit expansion} of integers, addition is computable by an
automaton:
\begin{lemma}[{\cite[Section 2.2.2.2]{cant10}}]\label{lem:sum}
  Let \(q \geq 2\) and \(B_q = \{-(q-1), \ldots, q-1\}\).  Words in
  \(B_q^{\;*}\) are interpreted as integers in base~\(q\).  The language of words
  over \(B_q \times B_q \times B_q\) such that the third component is the sum of the first
  two components is regular.  There is an automaton of size polynomial in \(q\)
  for that language.
\end{lemma}

\begin{restatable}{lemma}{thinreg}\label{lem:thinreg}
  Let \(\calA\) be an automaton over \(\BS(1,q)\), \(p, p'\) be states of \(\calA\), and
  \(k > 0\).  The sets \(\sv(\runs[k][p\to p']{\calA}), \sv(\retruns[k][p\to p']{\calA}),\) and
  \(\sv(\leftruns[k][p\to p']{\calA})\) are effectively regular.
\end{restatable}
\iflongversion
  \begin{proof}
    For simplicity, we deal with \emph{pointed expansions} of
    productions of runs, and indicate the easy changes that need to be
    made to deal with \emph{state views} of runs at the end of the
    proof.  As we draw intuition from two-way automata, we will assume
    that the positions along a run are always changing.  This is easily
    implemented by changing the alphabet to
    \(\Sigma = \{-1, 0, 1\} \times \{-1, 1\}\), and introducing
    intermediate states when translating \((1, 0)\) to, say,
    \((1, 1)(0, -1)\). This modification can turn runs that are $k$-thin
    into runs that are $2k$-thin: In addition to the $k$ state occurrences
    from the old run, one also sees at most $k$ state occurrences resulting
    from non-moving transitions one position to the right. This, however, is
    not an issue: We perform the construction below for thickness $2k$. Then it
    is obvious from our construction that it can be adapted to only capture
    those $2k$-thin runs in which each original state occurs at most $k$ times
    in each position.

    We will prove the statement in two steps.  First, we will convert \(\calA\) into
    an automaton that reads \(k\)-tuples of letters from \(\{-1, 0, 1\}\).%
    Each
    component corresponds to one of the ``threads''
    of a run of \(\calA\) at a given position in the input.  Second, we
    apply \Cref{lem:sum} to conclude that, based on the regular language over
    \(\{-1, 0, 1\}^k\) accepted by this new automaton, we can compute the componentwise sum in
    \(\Z[\tfrac{1}{q}]\).

    \emph{(Step 1: From \(\calA\) to \(k\)-component regular language.)}\quad This is akin
    to the classical proof~\cite{shepherdson1959reduction}
    that deterministic two-way automata can be turned into
    nondeterministic one-way automata.  Indeed, since the runs we are interested
    in are \(k\)-thin, we can follow \(k\) partial executions of \(\calA\), half from
    left to right, and half from right to left, and check that the reversals of direction
    are consistent.

    In more detail, we will build a nondeterministic automaton \(\calB\), whose
    set of states is \((Q_\calA\times \{L, R\})^{\leq k}\) and alphabet is
    \(\{-1, 0, 1\}^{\leq k}\).  Each component of a given state follows a portion of a
    \(k\)-thin run; it is thus expected that the letters \(L\) and \(R\), standing for
    left and right, and specifying the direction of the partial run, alternate from component
    to component.

    We now specify the transition relation of \(\calB\).  Let \(X\) and \(Y\) be two
    states of \(\calB\) of the same size \(\ell \leq k\):
    \[X = ((p_1, d_1), \ldots, (p_\ell, d_\ell)), \quad Y = ((p_1', d_1'), \ldots, (p_\ell',
    d_\ell'))\enspace.\]
    We add a transition between \(X\) and \(Y\) labeled \((a_1, \ldots, a_\ell)\) if
    for all \(i\):
    \begin{itemize}
    \item \mbox{}\(d_i = d'_i\),
    \item if \(d_i = R\), then \((p_i, (a_i, -1), p_i')\) is an edge in \(\calA\), and
    \item if \(d_i = L\), then \((p'_i, (a_i, 1), p_i)\) is an edge in \(\calA\).
    \end{itemize}
    These transitions check the consistency of a single step.  We also add
    transitions that correspond to the initial and final transitions of runs from
    \(p\) to \(p'\) in \(\calA\) (1 and 2 below), and transitions that check reversals
    (3 and 4 below):
    \begin{enumerate}
    \item At any time, \(\calB\) can take a transition on \(\varepsilon\) that either inserts
      \((p, R)\) as the first component of the current state, or removes \((p, L)\) in
      that component;
    \item At any time, \(\calB\) can take a transition on \(\varepsilon\) that either inserts
      \((p', L)\) in the last component of the current state, or removes \((p', R)\) in
      that component;
    \item At any time, \(\calB\) can take a transition on \(\varepsilon\) that inserts two
      components \((r, L)\) and \((r, R)\) within the current state, consecutively,
      for any state \(r\);
    \item At any time, \(\calB\) can take a transition on \(\varepsilon\) that removes two
      consecutive components of the form \((r, R)\) and \((r, L)\) from the current
      state, for any state \(r\).
    \end{enumerate}
    Naturally, this is subject to the constraint that a state has at most \(k\)
    components.  Finally, we set the empty vector as the initial and final state.

    To obtain the desired automaton for \(\runs[k][p\to p']{\calA}\), we additionally
    modify \(\calB\) so that transitions of type 1 and 2 are taken exactly once.
    Moreover, in transition 1, if \((p, R)\) is inserted, then the next symbol read
    is annotated with \(\bullet\); if \((p, L)\) is removed, then the previous symbol read
    is annotated with \(\bullet\).  Similarly, transition 2 annotates the next or previous
    symbol read with \(\triangleleft\).

    The automata for \(\retruns[k][p\to p']{\calA}\) and
    \(\leftruns[k][p\to p']{\calA}\) are obtained by a regular constraint
    on~\(\calB\): a simulated run is returning if the symbol annotated with
    \(\bullet\) is also annotated with \(\triangleleft\), and it is returning-left if this is the last
    symbol.

    \emph{(Step 2: Computing the addition.)}\quad This is a simple application of
    \Cref{lem:sum}, noting that we can keep the annotations \(\bullet\) and \(\triangleleft\) as is.

    \emph{(From pointed expansions to state views.)}\quad  The automaton \(\calB\)
    above actually knows the states in which the different partial runs of \(\calA\)
    are; this is what is stored in \(\calB\)'s states.  The alphabet of \(\calB\) can
    thus be extended to \((\{-1, 0, 1\} \times Q)^{\leq k}\), in such a way that each digit
    carries the information of the state in which it was emitted.  Then Step 2
    can be changed to not only compute the addition, but also produce the collection
    of all these states.
  \end{proof}
\fi

\secorpar{Iterations of returning-left thin cycles are PE-regular}\label{par:itcyc}

It is well-known that for every set $S\subseteq\N$ the generated
monoid $S^* = \{ s_1 + \dots + s_m \mid s_1, \ldots, s_m \in S, m \ge 0 \}$
is eventually identical with $\gcd(S)\cdot\N$. In other
words, the set $(\gcd(S)\cdot\N)\setminus S^*$ is finite and we may
define $F(S)=\max ((\gcd(S)\cdot\N) \setminus S^*)$.  The number $F(S)$
is called the \emph{Frobenius number} of $S$.
With this, we have
$S^* = \{n\in S^* \mid n\le F(S)\}\cup \{n\in\gcd(S)\cdot\N \mid n>F(S)\}$.
If $S\subseteq-\N$, then we set $F(S):=F(-S)$.
Now consider an arbitrary set $S \subseteq \Z$.
If $S$ contains both a
positive and a negative number, then $S^*=\gcd(S)\cdot\Z$ and we set
$F(S):=0$. We shall use the following well-known fact~\cite{Wilf78}.
\begin{lemma}\label{frobenius-finite-set}
  If $S=\{n_1,\ldots,n_k\}$ with $0 < n_1<\cdots< n_k$, then
  $F(S)\le n_k^2$.
\end{lemma}

\begin{lemma}\label{lem:star-left-runs}
  For every automaton $\calA$ over $\BS(1,q)$, the language
  $\pe([\leftruns[k][p\to p]{\calA}]^*)$ is effectively regular.
\end{lemma}
\begin{proof}
  Recall that we identify each $r\in\Z[\tfrac{1}{q}]$ with $(r,0)\in\Z[\tfrac{1}{q}]$.
  In particular, for $n\in\Z$, $\pe(n)$ is the same as $\pe((n,0))$.
  
  Denote $S=[\leftruns[k][p\to p]{\calA}]$. We first consider the case
  $S\subseteq\N$ and $S\ne\emptyset$.  Suppose we can compute $\gcd(S)$ and a bound
  $B\in\N$ with $B\ge F(S)$. Then we have
  \begin{equation}
    S^* = \underbrace{\{ n\in S^* \mid n\le B\}}_{=:X} ~\cup~
    \underbrace{\{n\in\gcd(S)\cdot \N \mid n>B\}}_{=:Y} \label{star-decomposition}
  \end{equation}
  and it
  suffices to show that $\pe(X)$ and $\pe(Y)$ are effectively
  regular. Note that $X$ is finite and can be computed by finding all $n\le B$ 
  with $n\in S$ (recall that membership in $S$ is decidable
  because $\sv(\leftruns[k][p\to p]{\calA})$ is effectively regular by \cref{lem:thinreg}) and
  building sums. Moreover, $\pe(Y)$ is regular because the set
  $L_0=\pe(\gcd(S)\cdot\N)$ is effectively regular and so is
  $L_1=\{\pe(n) \mid n\in\N,~n>B\}$, and hence $\pe(Y)=L_0\cap L_1$.
  
  Thus, it remains to compute $\gcd(S)$ and some $B\ge F(S)$. For
  the former, find any $r\in S$ and consider its decomposition
  $r=p_1^{e_1}\cdots p_m^{e_m}$ into prime powers. For each
  $i\in[1,m]$, we compute $d_i\in [0,e_i]$ and $n_i\in S$ such that
  (i)~$S\subseteq p_i^{d_i}\cdot \N$, and (ii)~
  $n_i\in S\setminus p_i^{d_i+1}\cdot\N$.  Since for $d\in\N$, we can
  construct an automaton for $\pe(S\cap d\cdot\N)$, these $d_i$ and
  $n_i$ can be computed. Observe that
  $\gcd(S)=p_1^{d_1}\cdots p_m^{d_m}$. Let
  $T=\{r,n_1,\ldots,n_k\}$. Observe that $\gcd(T)=\gcd(S)$, and hence
  $T^*$ and $S^*$ are ultimately identical. Since $T\subseteq S$, this
  means $F(S)\le F(T)$. By \cref{frobenius-finite-set}, we have
  $F(T)\le (\max\{r,n_1,\ldots,n_k\})^2$, which yields our bound $B$.

  The case $S\subseteq-\N$ is analogous to $S\subseteq\N$. If
  $S$ contains a positive and a negative number, then
  $S^*=\gcd(S)\cdot\Z$, so it suffices to just compute $\gcd(S)$. This
  is done as above.
  Finally, deciding between these three cases is easy. This completes the proof.
\end{proof}



\secorpar{Wrapping up: Proof of \cref{main-effective-regularity}}

Let \(\calA\) be an automaton over \(\BS(1, q)\) with state
set~\(Q\). \Cref{cor:decomp} indicates that the set of productions of
accepting runs is the same as the set of productions of \(k\)-thin runs in
which thin cycles are introduced.

By \Cref{lem:thinreg}, \(\sv(\runs[k]{\calA})\) is a regular language \(L\).
For any state \(p\) of \(\calA\), let \(L_p = \pe\big([\leftruns[k][p\to
p]{\calA}]^*\big)\), a regular language by \cref{lem:star-left-runs}.  For
padding purposes, let \(s \in Q\) be some state, and let \(h\) be the morphism from \((\Phi_q
\cup \{\pm\})^*\) to \((\Phi_q \cup Q \cup \{\pm\})^*\) defined, for any \(a \in \Phi_q\), by \(h(a)
= as^{|Q|}\), and \(h(+) = +\), \(h(-)=-\).  Define now \(L_p'\) to be the image by \(h\) of
the version of \(L_p\) where arbitrary \(0\)'s are added after the sign, and at
the end of the number (these 0's do not change the value represented).

Consider now the language \(R\) over the alphabet
\((\Phi_q \cup Q \cup \bar{Q} \cup \{\pm\})^{|Q| + 1}\) whose projection on the first
component is the language \(L\), and the other components correspond to the
languages \(L'_p\), for each \(p \in Q\).
The first component indicates in particular the states of
\(\calA\) that visited that location; to synchronize the different components of
\(R\), we ensure that the letter annotated with \(\bullet\) in \(L'_p\) is aligned with a
letter from \(L\) that is followed by \(p\)---that is, the starting position of
\(L'_p\) is at a position in \(L\) that is seen while being in the state \(p\).

Finally, an automaton can do the componentwise addition in base \(q\),
collapsing the \(|Q|+1\) components into a single one.
The radix point is given by the digit with \(\bullet\) of \(L\), i.e., in the first
component; and similarly for \(\triangleleft\).  The resulting language, thanks to
\Cref{cor:decomp}, is the language of the pointed expansions of all runs in
\(\runs{\calA}\).\hfill\qedsymbol

\section{Complexity}\label{complexity}

In this section, we prove
\cref{main-complexity,main-complexity-fixed}. 
For the upper bounds in
\cref{main-complexity,main-complexity-fixed}, we shall rely on the
fact that, given an element $g\in\Z[\tfrac{1}{q}]\rtimes\Z$ as a word
over $\Sigma=\{a,a^{-1},t,t^{-1}\}$, one can compute the
pointed expansion $\pe(g)$ in logarithmic space. This is a direct
consequence of a result of Elder, Elston, and
Ostheimer~\cite[Proposition~32]{ELDER2013260}. They show that given a
word $w$ over $\Sigma$, one can compute in logarithmic space an
equivalent word of one of the forms (i)~$t^i$,
(ii)~$(a^{\eta_0})^{t^{\alpha_0}}(a^{\eta_1})^{t^{\alpha_1}}\cdots
(a^{\eta_k})^{t^{\alpha_k}}t^i$ or
(iii)~$(a^{-\eta_0})^{t^{\alpha_0}}(a^{-\eta_1})^{t^{\alpha_1}}\cdots
(a^{-\eta_k})^{t^{\alpha_k}}t^i$, where $i\in\Z$, $k\in\N$,
$0<\eta_j<q$ for $j\in[0,k]$, and $\alpha_0>\cdots>\alpha_k$. Here,
$x^y$ stands for $y^{-1}xy$ in the group. Since these normal forms
denote the elements (i)~$(0,i)$,
(ii)~$(\sum_{j=0}^k \eta_jq^{-\alpha_j},i)$ and
(iii)~$(-\sum_{j=0}^k \eta_jq^{-\alpha_j},i)$, respectively, it is
easy to turn these normal forms into $\pe(w)$ using logarithmic
space.

This allows us to prove \cref{main-complexity-fixed}: For every
rational subset $R\subseteq\BS(1,q)$, the language $\pe(R)$ is a
regular language. In particular, there exists a deterministic
automaton $\calB$ for $\pe(R)$. Therefore, given $g\in\BS(1,q)$ as a
word over $\{a,a^{-1},t,t^{-1}\}$, we compute $\pe(g)$ in logspace and
then check membership of $\pe(g)$ in $\lang{\calB}$, which is
decidable in logarithmic space.

\newcommand{\varv}{\mathsf{v}}
\newcommand{\calS}{\mathcal{S}}

\iflongversion
  \subsection{PSPACE-hardness}
\fi
The $\PSPACE$ lower bound in
\cref{main-complexity} is a reduction from the intersection nonemptiness of finite-state
automata, a well-known
$\PSPACE$-complete problem~\cite{Kozen77}.
\iflongversion
  \begin{theorem}
    Rational subset membership is \(\PSPACE\)-hard.
  \end{theorem}
  \begin{proof}
    Let $q \ge 2$ be fixed.
    We give a reduction from the intersection nonemptiness problem
    for deterministic finite automata (DFA), a $\PSPACE$-hard problem~\cite{Kozen77}.
    Let $\D_1, \ldots, \D_n$, DFA over a finite alphabet $\Gamma$, $|\Gamma| \ge 2$,
    form an instance of that problem.
    We will describe an automaton \A over \mainbs that accepts the identity element
    of \mainbs if and only if there is a word $w \in \Gamma^*$ accepted by all $\D_i$.

    We first fix any injective mapping $f \colon \Gamma \to \{0, 1, \ldots, q-1\}^\ell$
    for $\ell = \lceil \log_2 |\Gamma| \rceil$.
    Transform $\D_1, \ldots, \D_n$ into nondeterministic
    finite automata (NFA) $\D'_1, \ldots, \D'_n$ over $\{0, 1, \ldots, q-1\}$ such that
    $\lang{\D'_i} = 1 \cdot f(\lang{\D_i}) \cdot 1$ for all~$i$.
    It is immediate that
    $\lang{\D_1} \cap \ldots \cap \lang{\D_n}$ is nonempty if and only if so is
    $\lang{\D'_1} \cap \ldots \cap \lang{\D'_n}$.

    We now describe the construction of the automaton~\A; it will be convenient
    for us to think of the input word as being written (\emph{produced}) rather
    than read by~\A.
    This word over $\{-1, 0, 1\} \times \{-1,1\} \subseteq \mainbs$
    corresponds to instructions
    to a machine working over an infinite tape with alphabet $\{0, 1, \ldots, q-1\}$,
    as per the intuition explained in \Cref{sec:prelim},
    and we will think of \A as moving left and right over that tape, updating
    the values in its cells.
    We emphasize that this tape is \emph{not} the input tape of \A,
    but instead corresponds to the actions of generators of \mainbs.

    The automaton \A will subdivide the tape into $n$~tracks.
    Suppose the cells of the tape are numbered, with indices $m \in \Z$;
    then the $i$th track consists of all cells with indices~$x$
    such that $x \equiv i \mod n$.
    The automaton \A will move left and right over the tape by
    producing $t = (0, 1)$ and $t^{-1} = (0, -1)$, two of the generators of \mainbs
    as monoid.
    Similarly,
    the current cell can be updated by producing $a = (1, 0)$ and $a^{-1} = (-1, 0)$,
    i.e., performing increments and decrements.
    The automaton will always remember in its finite-state memory which of the tracks
    the current cell belongs to.

    The workings of \A are as follows.
    It will enumerate $i = 1, \ldots, n$ one by one,
    and for each~$i$ it will guess and print some word accepted by the NFA~$\D'_i$
    on the $i$th track of the tape.
    (When we refer to guessing, this corresponds to the nondeterminism in the definition
    of automata over groups.)
    When incrementing~$i$, it will not only move to the $(i+1)$st track but also guess
    which specific cell in this track to move to. That is, in principle, \A may
    move arbitrarily far left or right over the tape.
    After all values of~$i$ have been enumerated,
    the automaton \A will guess some position of track~$1$ on the tape, moving
    to that position. Suppose the corresponding cell is numbered~$x \in \Z$,
    $x \equiv 1 \mod n$;
    then \A will transition to its \emph{final phase},
    performing the following sequence of operations:
    \begin{enumerate}
    \item\label{lb:removal}
      For $i = 1, \ldots, n$:
      perform decrement of the cell value once ($k = 1$ times), and then move to the adjacent cell
      with larger index (thus proceeding to track~$i+1$, or to track~$1$ again if $i=n$).
      \par
      We think of this \emph{sequence} of instructions as the \emph{removal} of $k$, $k = 1$.
    \item Perform the following operations in a loop, taken arbitrarily many times
      (terminating after some nondeterministically chosen iteration):
      \begin{itemize}
      \item Guess an element $g \in \{0, 1, \ldots, q-1\}$.
      \item \emph{Remove} $g$ (similarly to step~\ref{lb:removal}).
      \end{itemize}
    \item \emph{Remove} $1$ (as in step~\ref{lb:removal}).
    \item Move to an arbitrarily chosen cell of the tape and terminate
      (i.e., transition to a final state).
    \end{enumerate}
    We now claim that the final configuration of the tape can be all-$0$ (i.e.,
    the produced generators of \mainbs can yield the identity element of \mainbs)
    if and only if there is a word accepted by all machines $\D'_i$, $i = 1, \ldots, n$.

    Indeed, observe that, by the construction of \A, at the end of the simulation
    of NFA $\D'_1, \ldots, \D'_n$ each track~$i$ will contain a word
    of the form $1 \cdot f(w_i) \cdot 1$ where $w_i \in \lang{\D_i}$,
    with zeros all around it. The words written on different tracks may or may not be aligned
    with each other. Clearly, if all $w_i$ are chosen to be the same word, $w$, and
    the leftmost $1$s are all aligned with each other, then in the final phase
    of computation the automaton~\A can guess the word~$w$ and \emph{remove} it
    (or rather, remove $1 \cdot f(w) \cdot 1$) from
    the tape completely (with delimiters). After that, it can guess the location of
    cell~$0$ and move to that cell---this corresponds to the product of the produced
    generators being the identity of \mainbs.

    Therefore, it remains to see that the final phase cannot transform the tape
    configuration to all-$0$ unless all words $w_i$ are the same and the delimiting
    $1$s are aligned. But for this, it suffices to observe that the final phase
    (excepting the last operation)
    amounts, in terms of the group \mainbs, to subtracting a number of the following
    form (written in base~$q$):
    \begin{equation*}
      \underbrace{1 \ldots 1\vphantom{g_1}}_{n}\,
      \underbrace{g_1 \ldots g_1}_{n}
      \ldots
      \underbrace{g_s \ldots g_s}_{n}\,
      \underbrace{1 \ldots 1\vphantom{g_1}}_{n}
      \enspace,
    \end{equation*}
    where $s \in \N$ and $g_1, \ldots, g_s \in \{0, 1, \ldots, q-1\}$
    are chosen nondeterministically by \A.
    If the result of subtraction is $0 \in \Z$, then the content of the tape
    did indeed correspond to a number of this form. So the simulation of phase
    left each track with the same content, $\ldots 0 0 1 g_1 \ldots g_s 1 0 0 \ldots$\,,
    which means that $f^{-1}(g_1 \ldots g_s) \in \lang{\D_1} \cap \ldots \cap \lang{\D_n}$.

    Since the construction of the automaton \A can be performed in
    polynomial time (and even in logarithmic space), this completes the proof.
  \end{proof}
\fi

\iflongversion
  \subsection{PSPACE membership}
\fi
For the $\PSPACE$ upper bound,
we strengthen \cref{main-effective-regularity} by constructing a
polynomial-size representation of an exponential size automaton for
the resulting regular language.
%
A \emph{succinct finite
  automaton} is a tuple
$\calS=(n,\Gamma,(\varphi_x)_{x\in\Gamma\cup\{\varepsilon\}}, p_0,
p_f\})$, where $n\in\N$ is its \emph{bit length}, $\Gamma$ is its
\emph{input alphabet},
$\varphi_x(\varv_1,\ldots,\varv_n,\varv'_1,\ldots,\varv'_n)$ is a
formula from propositional logic with free variables
$\varv_1,\ldots,\varv_n,\varv'_1,\ldots,\varv'_n$ for each
$x\in\Gamma\cup\{\varepsilon\}$, $p_0\in\{0,1\}^n$ is its
\emph{initial state}, and $p_f\in\{0,1\}^n$ is its \emph{final state}.
The \emph{size} of $\calS$ is defined as
$|\calS|=n+\sum_{x\in\Gamma\cup\{\varepsilon\}} |\varphi_x|$, where
$|\varphi|$ denotes the length of the formula $\varphi$.

Moreover, $\calS$ \emph{represents} the automaton $\calA(\calS)$,
which is defined as follows. It has the state set $\{0,1\}^n$, initial
state $p_0$, and final state $p_f$. For states
$p=(b_1,\ldots,b_n), p'=(b'_1,\ldots,b'_n)\in\{0,1\}^n$ and
$x\in\Gamma\cup\{\varepsilon\}$, there is an edge $(p,x,q)$ in
$\calA(\calS)$ if and only if
$\varphi_x(b_1,\ldots,b_n,b'_1,\ldots,b'_n)$ holds. We define the
\emph{language accepted by $\calS$} as
$\lang{\calS}=\lang{\calA(\calS)}$.

We allow $\varepsilon$-edges in succinct automata, and with
Boolean formulas, one can encode steps in a Turing machine. Thus, a
succinct automaton of polynomial size can simulate a polynomial space
Turing machine with a one-way read-only input tape. Our descriptions
of succinct automata will therefore be in the style of polynomial
space algorithms. We show:
%
\begin{theorem}\label{effective-regularity-succinct}
  Given a rational subset $R\subseteq\BS(1,q)$, one can construct in
  polynomial space a polynomial-size
  succinct automaton accepting
  $\pe(R)$.
\end{theorem}
This allows us to decide rational subset membership in $\PSPACE$:
Given an automaton $\calA$ over $\BS(1,q)$ and an element $g$ as a
word over $\{a,a^{-1},t,t^{-1}\}$, we construct a succinct automaton
$\calB$ for $\pe(\lang{\calA})$ and the pointed expansion $\pe(g)$ in
logarithmic space. Since membership in succinct automata is well-known
to be in $\PSPACE$, we can check whether $\pe(g)\in\lang{\calB}$.


\subparagraph{Constructing succinct automata} It remains to prove
\cref{effective-regularity-succinct}.  The construction of a succinct
automaton for $\pe(R)$ proceeds with the same steps as in
\cref{rational-to-regular}. For most of these steps,
our constructions already yield small succinct automata
(e.g., one for $\pe([\leftruns[k][p \to p']{\calA}])$ in \cref{lem:thinreg}).
The exception is
\cref{lem:star-left-runs} --- in which case the key ingredient
is as follows.

\begin{proposition}\label{pspace-gcd-frobenius}
  Given an automaton $\calA$ over $\BS(1,q)$, a state $p$ of $\calA$, and $k\in\N$ in unary,
  one can compute in polynomial space the number
  $\gcd([\leftruns[k][p\to p]{\calA}])$ and a bound
  $B\ge F([\leftruns[k][p\to p]{\calA}])$. Both are at most
  exponential in $k$ and the size of $\calA$.
\end{proposition}

\noindent
Our bound on $F$ extends
the bound for automatic sets in $\N$~\cite[Lemma~4.5]{BellHS18}
to thin two-way computations.
Before proving \cref{pspace-gcd-frobenius}, let us show how it
implies \cref{effective-regularity-succinct}.
\begin{proof}[Proof of \cref{effective-regularity-succinct}]
  The constructions in \cref{lem:thinreg} and
  \cref{main-effective-regularity}, immediately yield a
  polynomial-size succinct automaton for $\pe(R)$ once a succinct
  automaton for each $\pe([\leftruns[k][p\to p]{\calA}]^*)$ is
  found. For the latter, we proceed as in \cref{lem:star-left-runs}.
  Let $S=[\leftruns[k][p\to p]{\calA}]$ and compute $\gcd(S)$ and a
  bound $B\ge F(S)$ using \cref{pspace-gcd-frobenius}. Then, by
  \cref{star-decomposition} on page~\pageref{star-decomposition}, it
  suffices to construct a succinct automaton for $\pe(X)$ and one for
  $\pe(Y)$. For $\pe(X)$, we use the fact that we can construct a
  succinct automaton $\calB$ for $\pe(S)$. Our automaton for $\pe(X)$
  proceeds as follows. With $\varepsilon$-transitions, it runs $\calB$
  to successively guess numbers $\le B$ from $S$ and stores each of
  them temporarily in its state. Such a number requires $O(\log(B))$
  bits. In another $O(\log(B))$ bits, it stores the sum of the numbers
  guessed so far. This continues as long as the sum is at most
  $B$. Then, our automaton reads the resulting sum from the
  input. This automaton clearly accepts $\pe(X)$.

  For $\pe(Y)$, we have to construct a succinct automaton that
  accepts any number $> B$ that is divisible by $\gcd(S)$. Since
  $\gcd(S)$ is available as a number with polynomially many digits, we
  can construct a succinct automaton accepting $\pe(\gcd(S)\cdot\N)$:
  It keeps the remainder modulo $\gcd(S)$ of the currently read
  prefix.  This requires $O(\log(\gcd(S))$ many bits.  Since $B$ also
  has polynomially many digits, we can construct a succinct
  automaton for $\{n\in\N \mid n>B\}$. An automaton for the
  intersection then accepts $\pe(Y)$.
\end{proof}

It is easy to see that the number produced by a returning-left run
is at most exponential in the length of the run. The exact bound will
not be important.
\begin{restatable}{lemma}{lengthVsMagnitude}\label{length-vs-magnitude}
  If $\rho$ is a run in $\leftruns[k]{\calA}$ of length $\ell$, then
  $|[\rho]|\le q^{2\ell}$.
\end{restatable}
\iflongversion
  \begin{proof}
    Let $m=\pmax(\rho)$.  Since $\rho$ is returning-left, $m$ can be at
    most $\ell/2$. Suppose in each position $i\in[0,m]$, $\rho$ adds
    $x_i\cdot q^{i}$. Then we have
    $|x_0|+\cdots+|x_m|\le \ell-2m$ and also
    \[ |[\rho]|=|x_0q^0+\cdots x_mq^m|\le |x_0|q^0+\cdots+|x_m|q^m. \]
    Under the condition $|x_0|+\cdots+|x_m|\le \ell-2m$, the expression
    on the right is clearly maximized for $x_m=\ell-2m$ and $x_i=0$ for
    $i\in[0,m-1]$. Therefore, we have $|[\rho]| \le (\ell-2m)
    q^m$. Since $\ell-2m\le q^\ell$, this implies
    $|[\rho]|\le (\ell-2m)q^m\le q^\ell\cdot q^{\ell/2}\le q^{2\ell}$.
  \end{proof}
\fi



The main ingredient for \cref{pspace-gcd-frobenius} will be
\cref{four-runs}.  We write $\rho\shorter \rho'$ if
$|\rho|<|\rho'|$. Moreover, for $d\in\Z$, we write
$\rho\shorter[d] \rho'$ if $\rho\shorter\rho'$ and for some
$\ell\in\Z$, we have $[\rho']=\ell\cdot[\rho]+d$.%
\begin{restatable}{lemma}{fourRuns}\label{four-runs}
  There is a polynomial $f$ such that the following holds.  Let $\calA$
  be an $n$-state automaton over $\BS(1,q)$ and let $p, p'$ be two states of $\calA$. Let
  $\rho_{11}\in\leftruns[k][p \to p']{\calA}$ with $|\rho_{11}|>f(n,k)$. There exist runs
  $\rho_{00}, \rho_{10}, \rho_{01}\in\leftruns[k][p \to p']{\calA}$ and
  $d\in\Z$ so that:
  \begin{equation} \begin{matrix}
      \rho_{01} & \shorter[d] & \rho_{11} \\
      \vshorter &        & \vshorter \\
      \rho_{00} & \shorter[d] & \rho_{10}
    \end{matrix}
    \label{four-runs-relations}
  \end{equation}
\end{restatable}
\noindent Here, one shows that a long run can be shortened
independently in two ways: Going left in the diagram
\labelcref{four-runs-relations}, and going down.  Shortening
the run by ``going left'' changes the production of the run by the same
difference, up to a factor $\ell$ that may differ in the
two rows.

\iflongversion
  In order to prove \cref{four-runs}, we first show a version of
  \cref{four-runs} that applies to returning runs that go far to the left
  and far to the right.  We first show some auxiliary lemmas:

  \begin{lemma}\label{simple-hill-cutting}
    Let $\calA$ be an $n$-state automaton over $\BS(1,q)$. For every
    $\tau\in\retruns[k]{\calA}$ with $\pmax(\tau)>n^2$ or
    $\pmin(\tau)<-n^2$, there is a run $\tau'\in\retruns[k]{\calA}$ with
    $|\tau'|<|\tau|$ and $\pmin(\tau')\ge \pmin(\tau)$.
    Moreover, $\tau'$ begins and ends in the same states as $\tau$.
  \end{lemma}

  \begin{proof}
    If we consider the effect of the actions of $\calA$ on the cursor,
    then the statement amounts to the following statement on $n$-state
    one-counter automata with a $\Z$-counter and a single zero test
    at the end of (accepted) runs:
    if along a run $\tau$ the counter goes (i) above $n^2$ or (ii) below $-n^2$, then
    there is a strictly shorter run $\tau'$ which begins and ends in the same states
    and in which the minimum value of the counter is at least the minimum value
    of the counter in $\tau$.
    This can be proved using the standard hill-cutting argument
    (see, e.g.,~\cite[Lemma~5]{EtessamiWY10};
    cf.~\cite[Proposition~7]{Latteux83} as well as~\cite{ChistikovCHPW19} and references therein):
    in scenario~(i) one can apply it
    to reduce the \emph{maximum} value of the counter whilst retaining
    the minimum value; and in scenario~(ii) one can increase
    the \emph{minimum} value whilst retaining the maximum one.
  \end{proof}

  The following is a consequence of \cref{simple-hill-cutting}.
  \begin{lemma}\label{thin-runs-shortest-run}
    There is a polynomial $f$ so that for every $n$-state automaton
    $\calA$ over $\BS(1,q)$ and every two states $p,p'$ of $\calA$,
    the shortest run in $\leftruns[k][p \to p']{\calA}$
    has length $\le f(n,k)$.
  \end{lemma}

  \begin{lemma}\label{length-vs-max}
    If $\rho$ is a run in $\leftruns[k]{\calA}$, then
    $|\rho|/k<\pmax(\rho)+1$.
  \end{lemma}
  \begin{proof}
    The run $\rho$ can visit at most $\pmax(\rho)+1$ distinct
    positions. But since $\rho$ is $k$-thin, it can visit each position
    at most $k$ times.  Since $\rho$ has $|\rho|$ moves, we have
    $|\rho|+1\le k(\pmax(\rho)+1)$ and thus $|\rho|<k(\pmax(\rho)+1)$,
    hence $|\rho|/k<\pmax(\rho)+1$.
  \end{proof}

  We now turn to our simpler version of \cref{four-runs}.   For runs $\rho,\rho'$ and $d\in\Z$, we write
  $\rho\shorterdiff[d]\rho'$ if $|\rho|<|\rho'|$ and $[\rho']=[\rho]+d$.
  \begin{lemma}\label{four-runs-simple}
    Let $\calA$ be an $n$-state automaton over $\BS(1,q)$ and let
    $p, p'$ be two states of $\calA$.
    Suppose $\rho_{11}\in\retruns[k][p \to p']{\calA}$ is such that $\pmax(\rho_{11})> n^2$
    and $\pmin(\rho_{11})<-n^2$. Then there are runs
    $\rho_{00},\rho_{10},\rho_{01},\rho_{11}\in\retruns[k][p \to p']{\calA}$ and a
    number $d\in\Z$ so that $\pmin(\rho)\ge \pmin(\rho_{11})$ for
    $\rho\in\{\rho_{00},\rho_{01},\rho_{10}\}$ and the following holds:
    \begin{equation}
      \begin{matrix}
        \rho_{01} & \shorterdiff[d] & \rho_{11} \\
        \vshorterdiff &        & \vshorterdiff \\
        \rho_{00} & \shorterdiff[d] & \rho_{10}
      \end{matrix}
      \label{four-runs-simple-relations}
    \end{equation}
  \end{lemma}
  \begin{proof}
    Since $\pmax(\rho_{11})>n^2$ and $\pmin(\rho_{11})< -n^2$, we can
    decompose $\rho_{11}=\sigma_1\tau_1\nu_1$ such that
    $\sigma_1,\nu_1\in\retruns[k]{\calA}$ and $\tau_1\in\leftruns[k]{\calA}$
    and $\pmax(\tau_1)>n^2$ and either $\pmin(\sigma_1)<-n^2$ or
    $\pmin(\nu_1)<-n^2$. (Note that none of $\sigma_1$, $\tau_1$, $\nu_1$
    needs to be a cycle.) Without loss of generality, we assume
    $\pmin(\nu_1)<-n^2$.

    According to \cref{simple-hill-cutting}, there are
    $\nu_0,\tau_0\in\retruns[k]{\calA}$ with $|\nu_0|<|\nu_1|$ and
    $|\tau_0|<|\tau_1|$ and $\pmin(\nu_0)\ge\pmin(\nu_1)$ and
    $\pmin(\tau_0)\ge\pmin(\tau_1)$. Since $\tau_1\in\leftruns[k]{\calA}$, this
    implies $\tau_0\in\leftruns[k]{\calA}$.
    Define
    \begin{align*}
      \rho_{01}&=\sigma_1\tau_0\nu_1 & \rho_{11}=\sigma_1\tau_1\nu_1 \\
      \rho_{00}&=\sigma_1\tau_0\nu_0 & \rho_{10}=\sigma_1\tau_1\nu_0
    \end{align*}
    Then with $d=[\tau_1]-[\tau_0]$, we have $[\rho_{1i}]=[\rho_{0i}]+d$ for $i=0$ and $i=1$.
  \end{proof}

  We are now prepared to prove \cref{four-runs}.

  \begin{proof}[Proof of \cref{four-runs}]
    Write $\rho=\rho_{11}$, let $f(n,k)=3k(n^2+1)$, and suppose
    $|\rho|>3k(n^2+1)$. Then $\pmax(\rho)+1 > 3(n^2+1)$ by \cref{length-vs-max},
    so $\pmax(\rho)\ge 3(n^2+1)$
    and, in particular, we can decompose
    $\rho=\sigma\tau\nu$ so that $\sigma$ is the shortest prefix of
    $\rho$ with $\pmax(\sigma)=2(n^2+1)$ and $\sigma\tau$ is the longest
    prefix of $\rho$ with $\pmax(\sigma\tau)=2(n^2+1)$.  Since
    $\pmax(\rho)\ge 3(n^2+1)$, we have $\pmax(\tau)>n^2$.  We distinguish
    two cases.
    \begin{enumerate}
    \item Suppose $\pmin(\tau)< -n^2$. Then \cref{four-runs-simple} yields
      runs $\tau_{00}$, $\tau_{01}$, and $\tau_{10}$ so that for some $d\in\Z$, we have
      \[
      \begin{matrix}
        \tau_{01} & \shorterdiff[d] & \tau_{11} \rlap{${}= \tau$} \\
        \vshorterdiff &        & \vshorterdiff \\
        \tau_{00} & \shorterdiff[d] & \tau_{10}
      \end{matrix}
      \]
      and $\pmin(\tau')\ge\pmin(\tau)$ for every
      $\tau'\in\{\tau_{00},\tau_{01},\tau_{10}\}$. We set $\rho_{ij}=\sigma\tau_{ij}\nu$. Then each $\rho_{ij}$ belongs to $\leftruns[k][p \to p']{\calA}$ and we even have
      \[
      \begin{matrix}
        \rho_{01} & \shorterdiff[d] & \rho_{11} \\
        \vshorterdiff &        & \vshorterdiff \\
        \rho_{00} & \shorterdiff[d] & \rho_{10}
      \end{matrix}
      \]
      which implies \cref{four-runs-relations}.
      
    \item Suppose $\pmin(\tau)\ge -n^2$.  In this case,
      \cref{simple-hill-cutting} yields a run $\tau'\in\retruns[k]{\calA}$
      with $|\tau'|<|\tau|$ and $\pmin(\tau')\ge \pmin(\tau)$.
      
      Since now $\pmin(\tau)\ge -n^2$ and $\pmin(\tau')\ge -n^2$,
      we can decompose $\sigma=\sigma_1\sigma_2\sigma_3$ and
      $\nu=\nu_3\nu_2\nu_1$ so that
      \begin{itemize}
      \item $|\sigma_2|>0$ and $|\nu_2|>0$ and
      \item $\pos{\sigma_1} + \pos{\nu_1} = 0$ and
        $\pos{\sigma_2} + \pos{\nu_2} = 0$.
      \item $\sigma_1\sigma_3\tau\nu_3\nu_1$ and
        $\sigma_1\sigma_3\tau'\nu_3\nu_1$ again belong to
        $\leftruns[k]{\calA}$.
      \end{itemize}
      For ease of notation, we write $\tau_1=\tau$ and $\tau_0=\tau'$.
      We define
      \begin{align*}
        \rho_{01}&=\sigma_1\sigma_3\tau_1\nu_3\nu_1 & \rho_{11}&=\sigma_1\sigma_2\sigma_3\tau_1\nu_3\nu_2\nu_1 \\
        \rho_{00}&=\sigma_1\sigma_3\tau_0\nu_3\nu_1 & \rho_{10}&=\sigma_1\sigma_2\sigma_3\tau_0\nu_3\nu_2\nu_1 
      \end{align*}
      (where $\rho_{11}$ is repeated just for illustration). Then
      clearly the length relationships claimed in
      \cref{four-runs-relations} are satisfied.  Let
      $h_1=\pos{\sigma_1}$ and $h_2=\pos{\sigma_2}$. Then both for $i=0$ and for $i=1$, we have
      \begin{align*}
        [\rho_{1i}]&=[\sigma_1]+q^{h_1}[\sigma_2]+q^{h_1+h_2}[\sigma_3\tau_i\nu_3]+q^{h_1+h_2}[\nu_2]+q^{h_1}[\nu_1] \\
        [\rho_{0i}]&=[\sigma_1]+q^{h_1}[\sigma_3\tau_i\nu_3]+q^{h_1}[\nu_1].
      \end{align*}
      Therefore, with $d=(1-q^{h_2})[\sigma_1]+q^{h_1}[\sigma_2]+q^{h_1+h_2}[\nu_2]+q^{h_1}(1-q^{h_2})[\nu_1]$, we have
      \begin{align*}
        [\rho_{1i}]&=q^{h_2}[\rho_{0i}]+[\sigma_1]-q^{h_2}[\sigma_1]+q^{h_1}[\sigma_2]+q^{h_1+h_2}[\nu_2]+q^{h_1}[\nu_1]-q^{h_1+h_2}[\nu_1] \\
                &= q^{h_2}[\rho_{0i}] + d
      \end{align*}
      This means that indeed $\rho_{01}\shorter[d]\rho_{11}$ and $\rho_{00}\shorter[d]\rho_{10}$. \qedhere
    \end{enumerate}
  \end{proof}
\fi
\cref{thin-runs-small-non-divisible} applies
\cref{four-runs} to construct small numbers in $[\leftruns[k]{\calA}]$
that are not divisible by a given $m$. Later, these numbers allow us to
compute $\gcd([\leftruns[k][p\to p]{\calA}])$ and bound
$F([\leftruns[k][p\to p]{\calA}])$.
\begin{lemma}\label{thin-runs-small-non-divisible}
  There is a polynomial $f$ such that the following holds. Let $m\in\Z$.
  Let $\calA$ be an $n$-state automaton over $\BS(1,q)$ and let $p,p'$ be two states of $\calA$.
  Suppose
  there is a number in $[\leftruns[k][p \to p']{\calA}]$ not divisible by~$m$;
  then there is also an $s\in[\leftruns[k][p \to p']{\calA}]$ not divisible by~$m$ such that
  $|s|\le q^{f(n,k)}$.
\end{lemma}
\begin{proof}
  Let $f$ be the polynomial from \cref{four-runs}. Let
  $\rho\in\leftruns[k][p \to p']{\calA}$ be of minimal length such that
  $m$ does not divide $[\rho]$. Suppose
  $|\rho|>f(n,k)$.  Write $\rho_{11}=\rho$ and apply
  \cref{four-runs}. By minimality of $\rho_{11}$, we get
  $[\rho_{00}]\equiv [\rho_{10}]\equiv [\rho_{01}] \equiv 0\bmod{m}$.
  In particular, $\rho_{00}\shorter[d]\rho_{10}$ implies
  $d\equiv 0\bmod{m}$.  However, since $\rho_{01}\shorter[d]\rho_{11}$
  and $[\rho_{11}]\not\equiv 0\bmod{m}$, we get
  $d\not\equiv 0\bmod{m}$, a contradiction. Hence,
  $|\rho|\le f(n,k)$ and thus $|[\rho]|\le q^{2f(n,k)}$ by
  \cref{length-vs-magnitude}.
\end{proof}
With \cref{thin-runs-small-non-divisible} in hand, one can show
\cref{pspace-gcd-frobenius} similarly to \cref{lem:star-left-runs}.

\iflongversion
  \begin{proof}[Proof of \cref{pspace-gcd-frobenius}]
    Denote $S=[\leftruns[k][p\to p]{\calA}]$ and suppose
    $S\ne\emptyset$. Let $f_1$ be the polynomial from
    \cref{thin-runs-shortest-run}.  Then the shortest run in
    $\leftruns[k][p\to p]{\calA}$ has length $\le f_1(n,k)$. We can
    therefore guess a run $\rho$ of length $\le f_1(n,k)$.

    If we write $r=[\rho]$, then $|r|\le q^{2f_1(n,k)}$ by
    \cref{length-vs-magnitude}.  We can thus compute $r$ in polynomial
    space. Note that $g=\gcd(S)$ divides $r$ and thus
    $g\le q^{2f_1(n,k)}$. Let us now describe how to compute $g$ and a
    bound $B\ge F(S)$.

    We first consider the case $S\subseteq\N$. We compute the
    decomposition $r=p_1^{e_1}\cdots p_m^{e_m}$ into prime powers.  Note
    that each $e_i$ is at most polynomial. For each $i\in[1,m]$, there
    exists a $d_i\in[0,e_i]$ such that $S\subseteq p_i^{d_i}\cdot\N$ but
    $S\not\subseteq p_i^{d_i+1}\cdot\N$. We can compute $d_i$ in
    polynomial space, because we can construct a succinct finite
    automaton for $\pe(S)$ and, for every polynomially bounded $\ell$, we
    can construct a succinct automaton for
    $\pe(\N\setminus p_i^{\ell}\cdot\N)$: The latter keeps a remainder
    modulo $p_i^{\ell}$ in its state, accepting if this remainder is
    non-zero.  Thus, given a candidate $d_i$, we can construct a
    succinct automaton for $\pe(S\cap (\N\setminus p_i^{d_i}\cdot\N))$
    and one for $\pe(S\cap (\N\setminus p_i^{d_i+1}\cdot\N))$ and
    verify in $\PSPACE$ that the former is empty and the latter is not.
    Observe that now $\gcd(S)=p_1^{d_1}\cdots p_m^{d_i}$, meaning we can
    compute $\gcd(S)$ with polynomially many bits.

    We now compute a bound $B\ge F(S)$.  Let $f_2$ be the polynomial
    from \cref{thin-runs-small-non-divisible}. Since
    $S\cap (\N\setminus p_i^{d_i+1}\cdot\N)$ is non-empty,
    \cref{thin-runs-small-non-divisible} tells us that there is a number
    $n_i\in S\cap (\N\setminus p^{d_i+1}\cdot\N)$ with
    $n_i\le q^{f_2(n,k)}$. We can therefore guess a number $n_i\in\N$
    with polynomially many digits and verify that
    $n_i\in S\cap (\N\setminus p_i^{d_i+1}\cdot\N)$.
    
    Since $p_i^{d_i+1}$ does not divide $n_i\in S$, we know that the set
    $T=\{r,n_1,\ldots,n_m\}$ satisfies
    $\gcd(T)=p_1^{d_1}\cdots p_m^{d_m}=\gcd(S)$. Therefore, the sets
    $T^*$ and $S^*$ are ultimately identical. Since trivially
    $T^*\subseteq S^*$, we may conclude $F(S)\le F(T)$. Moreover,
    according to \cref{frobenius-finite-set}, we have
    $F(T)\le (\pmax\{r,n_1,\ldots,n_m\})^2$ and we set
    $B:=(\pmax\{r,n_1,\ldots,n_m\})^2$.  Since $r\le q^{2f_1(n,k)}$ and
    $n_i\le q^{f_2(n,k)}$, we know that $B$ is at most
    $q^{4f_1(n,k)+2f_2(n,k)}$ and can clearly be computed from $r$,
    $n_1,\ldots,n_m$. This completes the case $S\subseteq\N$.

    In the case $S\subseteq-\N$, we can proceed analogously. If $S$
    contains a positive number and a negative number, we compute
    $\gcd(S)$ as above (replacing $\N$ with $\Z$) and can set $B=0$
    because $F(S)=0$.
    %
    %
  \end{proof}
\fi


\section{Recognizability}\label{recognizability}
In this section, we prove \cref{main-recognizability}.  We first
present a characterization of recognizability that is easily checkable
for PE-regular subsets. It is well-known that a subset $S$ of $\Z$ is
recognizable if and only if there is a $k\in\Z\setminus\{0\}$ such
that for every $s\in\Z$, we have $s\in S$ if and only if $s+k\in S$.
Our characterization is an analog for Baumslag-Solitar groups.

A subset $S\subseteq\Z[\tfrac{1}{q}]\rtimes\Z$ 
is called \emph{$k$-periodic} if for every $s\in\Z[\tfrac{1}{q}]\rtimes\Z$, we have (i)~$s\in S$ if and only if $s(0,k)\in S$ and (ii)~for every $\ell\in\Z$, we have $s\in S$ if and only if $s(q^{\ell}-q^{\ell+k},0)\in S$.
In other words, membership in $S$ is insensitive to (i)~moving the
cursor $k$ positions and (ii)~replacing a power of $q$ by another
power of $q$ whose exponent differs by $k$. The set $S$ is
\emph{periodic} if it is $k$-periodic for some $k\ge 1$. We show the
following:
\begin{restatable}{proposition}{recognizableIffPeriodic}\label{recognizable-iff-periodic}
  A subset $S\subseteq\Z[\tfrac{1}{q}]\rtimes\Z$ is recognizable if
  and only if $S$ is periodic.
\end{restatable}
\iflongversion
  \begin{proof}
    Recall that $S$ is $k$-periodic if
    \begin{align}
      s\in S \iff s(0,k)\in S ~~~~\text{and}~~~~s\in S\iff s(q^\ell-q^{\ell+k},0)\in S~\text{for every $\ell\in\Z$}.\label{k-periodic}
    \end{align}
    Suppose $S$ is recognizable with a morphism
    $\varphi\colon \Z[\tfrac{1}{q}]\rtimes\Z\to K$ for some finite group
    $K$. Then there must be some $k\in\Z\setminus\{0\}$ with $\varphi((0,k))=1$:
    Otherwise, the map $\Z\to K$, $m\mapsto \varphi((0,m))$ would be injective,
    which is impossible for finite $K$. Now $\varphi((0,k))=1$ implies that
    $s\in S$ if and only if $s(0,k)\in S$ and thus the left equivalence in \cref{k-periodic}. Moreover, since
    \[ (q^\ell-q^{\ell+k},0)= (q^{\ell},0)(0,k)(-q^{\ell},0)(0,-k),\]
    we
    have $\varphi((q^{\ell}-q^{\ell+k},0))=1$ and hence $S$ satisfies
    the right equivalence in \cref{k-periodic}.  Thus $S$ is
    $k$-periodic.

    Suppose $S$ is $k$-periodic for $k\ge 1$ and consider the subgroup
    $H$ of $G=\Z[\tfrac{1}{q}]\rtimes\Z$ generated by $(0,k)$ and by
    $(q^\ell-q^{\ell+k},0)$ for all $\ell\in\Z$. We claim that $H$ is
    normal and the quotient $G/H$ is finite.  For normality, we have to
    check that for every generator $h$ of $H$ and every generator $g$ of
    $G$, we have $ghg^{-1}\in H$. Since $G$ is generated by $(1,0)$ and
    $(0,1)$, we have to consider the following cases:
    \begin{itemize}
    \item Let $h=(0,k)$ and $g=(1,0)$. Then $ghg^{-1}=(1,0)(0,k)(-1,0)=(q-q^k,0)$.
    \item Let $h=(0,k)$ and $g=(0,1)$. Then $ghg^{-1}=(0,1)(0,k)(0,-1)=(0,k)$.      
    \item Let $h=(q^{\ell}-q^{\ell+k},0)$ and $g=(1,0)$. Then $ghg^{-1}=(1,0)(q^{\ell}-q^{\ell+k},0)(-1,0)=(q^{\ell}-q^{\ell+k},0)$.
    \item Let $h=(q^{\ell}-q^{\ell+k},0)$ and $g=(0,1)$. Then $ghg^{-1}=(0,1)(q^{\ell}-q^{\ell+k},0)(0,-1)=(q^{\ell+1}-q^{\ell+1+k},0)$.        
    \end{itemize}
    In each case, $ghg^{-1}$ clearly belongs to $H$, hence $H$ is normal.

    We may therefore consider the quotient group $G/H$ and the
    projection $\pi\colon G\to G/H$.  Note that since $S$ is
    $k$-periodic, we know that for $s\in S$ if and only if $sh\in S$ for
    any $s\in S$ and $h\in H$. Therefore, if $\pi(s)=\pi(s')$, then
    $s\in S$ if and only $s'\in S$. Thus, $S$ is recognized by the morphism $\pi$
    and it suffices to show that $G/H$ is finite.

    We prove this by showing that for any $(\frac{p}{q^\ell},m)\in G$,
    we can multiply elements from $H$ to obtain an element $(r,n)$ with
    $r\in\{0,\pm 1,\pm 2,\ldots,\pm q^{k}-1\}$ and
    $n\in\{0,1,\ldots,k-1\}$. Since there are only finitely many
    elements of the latter shape, this clearly implies finiteness of
    $G/H$. We do this in three steps. We first transform the left
    component into a natural number. Then we turn the left component
    into a number in $\{0,\pm 1,\pm 2,\ldots,\pm q^{k}-1\}$. Finally. we
    bring the right component to a number in $\{0,\ldots,k-1\}$.

    For the first step, consider the element
    $g=(\frac{p}{q^{\ell}},m)\in G$.  By multiplying
    $(-q^{-m-\ell}+q^{-m-\ell+k},0)^p\in H$ to $g$, we obtain
    $(\frac{p}{q^{\ell-k}},m)$. If we repeat this, we end up with an
    element $(p,m)$ with $p\in\Z$ and $m\in\Z$.

    For the second step, consider $(p,m)\in G$ with $p,m\in\Z$. If
    $p\ge q^k$, we multiply with $(1-q^k,0)$ and obtain $(p+1-q^k,0)$,
    where $p+1-q^{k}<p$ (because $k\ge 1$). By repeating this, we end up
    at an element $(p,m)$ with $0\le p<q^k$.  In the case $p<-q^{k}$, we
    just multiply $(-q+q^{k}, 0)=(q-q^k,0)^{-1}$ instead of
    $(q-q^k,0)$. Thus, in general, we obtain an element $(p,m)$ with
    $p\in\{0,\pm 1,\pm 2,\ldots,\pm q^k-1\}$.

    For the third step, we merely reduce the right component modulo $k$:
    By multiplying $(0,k)$ or $(0,-k)$, we can clearly obtain an element
    $(p,m)$ where $m\in\{0,1,\ldots,k-1\}$ and where still $p\in\{0,\pm
    1,\pm 2,\ldots,\pm q^k-1\}$. Thus
    $G/H$ is is finite and recognizability of $S$ follows.
  \end{proof}
\fi

To decide whether a PE-regular $R\subseteq\BS(1,q)$ is recognizable, we
show effective regularity of the set $N\subseteq\{\ltr{a}\}^*$ of
all words $\ltr{a}^k$ such that $R$ is \emph{not} $k$-periodic. Then,
we just have to check whether $N$ contains all words $\ltr{a}^k$ with $k\ge 1$, which is clearly decidable. Since $R$ is PE-regular, the set
$D=R(G\mathord{\setminus} R)^{-1}\cup (G\mathord{\setminus} R)R^{-1}$
is effectively PE-regular (\cref{main-closure-properties}). Then $R$ is
not $k$-periodic if and only if $(0,k)\in D$ or
$(q^{\ell}-q^{\ell+k},0)\in D$ for some $\ell\in\Z$. The element
$(0,k)$ has the pointed expansion $\cursor{0}0^{k-1}\rpoint{0}$.  The
pointed expansions of $(q^\ell-q^{\ell+k},0)$ for $\ell\in\Z$ are
exactly those words obtained from words $-0^{r}(q-1)^{k-1}0^{s}$ for
$r,s\in\N$ by decorating one of the digits with $\cursor{}$ and with
$\rpoint{}$, and removing leading or trailing $0$'s. Therefore, it is
easy to see that $T_1=\{(\cursor{0}0^{k-1}\rpoint{0}, \ltr{a}^k) \mid k\ge 1\}$
and
$T_2=\{ (\pe((q^{\ell}-q^{\ell+k},0)), \ltr{a}^k) \mid
\ell\in\Z,~k\ge 1\}$ are rational transductions.  This implies that
$N=T_1(\pe(D))\cup T_2(\pe(D))\subseteq\ltr{a}^*$
is effectively regular. Then clearly,
$R$ is not $k$-periodic if and only if $\ltr{a}^k\in
N$.


\bibliographystyle{plainurl}
\bibliography{bibliography}

\begin{thebibliography}{10}

\bibitem{DBLP:journals/corr/AubrunK13}
Nathalie Aubrun and Jarkko Kari.
\newblock Tiling problems on {Baumslag-Solitar} groups.
\newblock In {\em Proceedings of Machines, Computations and Universality 2013
  ({MCU} 2013)}, pages 35--46, 2013.
\newblock \href {https://doi.org/10.4204/EPTCS.128.12}
  {\path{doi:10.4204/EPTCS.128.12}}.

\bibitem{BaSi2010}
Laurent Bartholdi and Pedro~V. Silva.
\newblock Rational subsets of groups.
\newblock {\em CoRR}, abs/1012.1532, 2010.
\newblock Chapter 23 of the handbook AutoMathA (to appear).
\newblock \href {http://arxiv.org/abs/1012.1532} {\path{arXiv:1012.1532}}.

\bibitem{baumslag1962some}
Gilbert Baumslag and Donald Solitar.
\newblock Some two-generator one-relator non-{Hopfian} groups.
\newblock {\em Bulletin of the American Mathematical Society}, 68(3):199--201,
  1962.
\newblock \href {https://doi.org/10.1090/S0002-9904-1962-10745-9}
  {\path{doi:10.1090/S0002-9904-1962-10745-9}}.

\bibitem{bazhenova2000rational}
Galina~Aleksandrovna Bazhenova.
\newblock Rational sets in finitely generated nilpotent groups.
\newblock {\em Algebra and Logic}, 39(4):215--223, 2000.
\newblock \href {https://doi.org/10.1007/BF02681647}
  {\path{doi:10.1007/BF02681647}}.

\bibitem{BellHS18}
Jason~P. Bell, Kathryn Hare, and Jeffrey Shallit.
\newblock When is an automatic set an additive basis?
\newblock {\em Proceedings of the American Mathematical Society, Series B},
  5(6):50--63, 2018.
\newblock \href {https://doi.org/10.1090/bproc/37}
  {\path{doi:10.1090/bproc/37}}.

\bibitem{Benois1969}
Mich{\`e}le Benois.
\newblock Parties rationnelles du groupe libre.
\newblock {\em CR Acad. Sci. Paris}, 269:1188--1190, 1969.

\bibitem{Berstel1979}
Jean Berstel.
\newblock {\em Transductions and Context-Free Languages}.
\newblock Teubner, 1979.

\bibitem{cant10}
Valérie Berthé and Michel Rigo, editors.
\newblock {\em Combinatorics, automata, and number theory}, volume 135 of {\em
  Encyclopedia of Mathematics and its Applications}.
\newblock Cambridge University Press, 2010.

\bibitem{ChistikovCHPW19}
Dmitry Chistikov, Wojciech Czerwi\'{n}ski, Piotr Hofman, Michal Pilipczuk, and
  Michael Wehar.
\newblock Shortest paths in one-counter systems.
\newblock {\em Logical Methods in Computer Science}, 15(1), 2019.
\newblock \href {https://doi.org/10.23638/LMCS-15(1:19)2019}
  {\path{doi:10.23638/LMCS-15(1:19)2019}}.

\bibitem{DBLP:conf/icalp/CiobanuE19}
Laura Ciobanu and Murray Elder.
\newblock Solutions sets to systems of equations in hyperbolic groups are
  {EDT0L} in {PSPACE}.
\newblock In {\em Proceedings of the 46th International Colloquium on Automata,
  Languages, and Programming ({ICALP} 2019)}, pages 110:1--110:15, 2019.
\newblock \href {https://doi.org/10.4230/LIPIcs.ICALP.2019.110}
  {\path{doi:10.4230/LIPIcs.ICALP.2019.110}}.

\bibitem{Delgado2017a}
Jordi Delgado~Rodr\'{\i}guez.
\newblock {\em Extensions of free groups: algebraic, geometric, and algorithmic
  aspects}.
\newblock PhD thesis, Universitat Polit\`{e}cnica de Catalunya. Facultat de
  Matem\`{a}tiques i Estad\'{\i}stica, 2017.

\bibitem{DIEKERT2005105}
Volker Diekert, Claudio Gutierrez, and Christian Hagenah.
\newblock The existential theory of equations with rational constraints in free
  groups is {PSPACE} -complete.
\newblock {\em Information and Computation}, 202(2):105 -- 140, 2005.
\newblock \href {https://doi.org/10.1016/j.ic.2005.04.002}
  {\path{doi:10.1016/j.ic.2005.04.002}}.

\bibitem{DBLP:journals/dagstuhl-reports/DiekertKLM19}
Volker Diekert, Olga Kharlampovich, Markus Lohrey, and Alexei~G. Myasnikov.
\newblock {Algorithmic Problems in Group Theory (Dagstuhl Seminar 19131)}.
\newblock {\em Dagstuhl Reports}, 9(3):83--110, 2019.
\newblock \href {https://doi.org/10.4230/DagRep.9.3.83}
  {\path{doi:10.4230/DagRep.9.3.83}}.

\bibitem{DBLP:journals/ijac/DiekertL11}
Volker Diekert and J{\"{u}}rn Laun.
\newblock On computing geodesics in {Baumslag-Solitar} groups.
\newblock {\em International Journal on Algebra and Computation},
  21(1-2):119--145, 2011.
\newblock \href {https://doi.org/10.1142/S0218196711006108}
  {\path{doi:10.1142/S0218196711006108}}.

\bibitem{DBLP:conf/latin/DiekertMW14}
Volker Diekert, Alexei~G. Myasnikov, and Armin Wei{\ss}.
\newblock Conjugacy in {Baumslag}'s group, generic case complexity, and
  division in power circuits.
\newblock In {\em Proceedings of 11th Latin American Symposium on Theoretical
  Informatics ({LATIN} 2014)}, pages 1--12, 2014.
\newblock \href {https://doi.org/10.1007/978-3-642-54423-1\_1}
  {\path{doi:10.1007/978-3-642-54423-1\_1}}.

\bibitem{DBLP:journals/corr/abs-1910-02302}
Volker Diekert, Igor Potapov, and Pavel Semukhin.
\newblock Decidability of membership problems for flat rational subsets of
  {$\GL(2, \mathbb{Q})$} and singular matrices, 2019.
\newblock \href {http://arxiv.org/abs/1910.02302} {\path{arXiv:1910.02302}}.

\bibitem{DudkinTreyer2018}
F.~A. Dudkin and A.~V. Treyer.
\newblock Knapsack problem for {Baumslag–Solitar} groups.
\newblock {\em Siberian Journal of Pure and Applied Mathematics}, 18:43--55,
  2018.
\newblock \href {https://doi.org/10.33048/pam.2018.18.404}
  {\path{doi:10.33048/pam.2018.18.404}}.

\bibitem{elder2010}
Murray Elder.
\newblock A linear-time algorithm to compute geodesics in solvable
  {Baumslag–Solitar} groups.
\newblock {\em Illinois Journal of Mathematics}, 54(1):109--128, 2010.
\newblock \href {https://doi.org/10.1215/ijm/1299679740}
  {\path{doi:10.1215/ijm/1299679740}}.

\bibitem{ELDER2013260}
Murray Elder, Gillian Elston, and Gretchen Ostheimer.
\newblock On groups that have normal forms computable in logspace.
\newblock {\em Journal of Algebra}, 381:260 -- 281, 2013.
\newblock \href {https://doi.org/10.1016/j.jalgebra.2013.01.036}
  {\path{doi:10.1016/j.jalgebra.2013.01.036}}.

\bibitem{EtessamiWY10}
Kousha Etessami, Dominik Wojtczak, and Mihalis Yannakakis.
\newblock Quasi-birth-death processes, tree-like {QBDs}, probabilistic
  1-counter automata, and pushdown systems.
\newblock {\em Perform. Eval.}, 67(9):837--857, 2010.

\bibitem{GinsburgSpanier1966a}
Seymour Ginsburg and Edwin~H. Spanier.
\newblock Bounded regular sets.
\newblock {\em Proceedings of the American Mathematical Society},
  17(5):1043--1049, 1966.
\newblock \href {https://doi.org/10.2307/2036087} {\path{doi:10.2307/2036087}}.

\bibitem{Ibarra1978}
Oscar~H. Ibarra.
\newblock Reversal-bounded multicounter machines and their decision problems.
\newblock {\em Journal of the ACM}, 25(1):116--133, 1978.
\newblock \href {https://doi.org/10.1145/322047.322058}
  {\path{doi:10.1145/322047.322058}}.

\bibitem{KAPOVICH2002608}
Ilya Kapovich and Alexei Myasnikov.
\newblock Stallings foldings and subgroups of free groups.
\newblock {\em Journal of Algebra}, 248(2):608 -- 668, 2002.
\newblock \href {https://doi.org/10.1006/jabr.2001.9033}
  {\path{doi:10.1006/jabr.2001.9033}}.

\bibitem{KharlampovichLopezMiasnikov2019}
Olga Kharlampovich, Laura L\'{o}pez, and Alexei Miasnikov.
\newblock Diophantine problem in some metabelian groups, 2019.
\newblock \href {http://arxiv.org/abs/1903.10068} {\path{arXiv:1903.10068}}.

\bibitem{Kozen77}
Dexter Kozen.
\newblock Lower bounds for natural proof systems.
\newblock In {\em Proceedings of the 18th Annual Symposium on Foundations of
  Computer Science (FOCS 1977)}, pages 254--266, 1977.
\newblock \href {https://doi.org/10.1109/SFCS.1977.16}
  {\path{doi:10.1109/SFCS.1977.16}}.

\bibitem{Latteux83}
Michel Latteux.
\newblock Langages {\`{a}} un compteur.
\newblock {\em J. Comput. Syst. Sci.}, 26(1):14--33, 1983.

\bibitem{Lohrey2016survey}
Markus Lohrey.
\newblock The rational subset membership problem for groups: a survey.
\newblock In C.~M. Campbell, M.~R. Quick, E.~F. Robertson, and C.~M.
  Roney-Dougal, editors, {\em Groups St Andrews 2013}, volume 422 of {\em Lond.
  Math. S.}, pages 368--389, Cambridge, United Kingdom, 2016. Cambridge
  University Press.
\newblock \href {https://doi.org/10.1017/CBO9781316227343.024}
  {\path{doi:10.1017/CBO9781316227343.024}}.

\bibitem{DBLP:journals/ijac/LohreyS08}
Markus Lohrey and G{\'{e}}raud S{\'{e}}nizergues.
\newblock Rational subsets in {HNN}-extensions and amalgamated products.
\newblock {\em International Journal on Algebra and Computation},
  18(1):111--163, 2008.
\newblock \href {https://doi.org/10.1142/S021819670800438X}
  {\path{doi:10.1142/S021819670800438X}}.

\bibitem{LohreyZetzsche2020a}
Markus Lohrey and Georg Zetzsche.
\newblock Knapsack in metabelian {Baumslag-Solitar} groups, 2020.
\newblock \href {http://arxiv.org/abs/2002.03837} {\path{arXiv:2002.03837}}.

\bibitem{MyNiUs14}
Alexei Myasnikov, Andrey Nikolaev, and Alexander Ushakov.
\newblock Knapsack problems in groups.
\newblock {\em Mathematics of Computation}, 84:987--1016, 2015.
\newblock \href {https://doi.org/10.1090/S0025-5718-2014-02880-9}
  {\path{doi:10.1090/S0025-5718-2014-02880-9}}.

\bibitem{Robinson1993}
David Robinson.
\newblock {\em Parallel Algorithms for Group Word Problems}.
\newblock PhD thesis, Department of Mathematics, University of Califoria, San
  Diego, 1993.

\bibitem{Romanovskii1974a}
N.~S. Romanovski\u{\i}.
\newblock Some algorithmic problems for solvable groups.
\newblock {\em Algebra and Logic}, 13:13--16, 1974.
\newblock \href {https://doi.org/10.1007/BF01462922}
  {\path{doi:10.1007/BF01462922}}.

\bibitem{Romanovskii1980a}
N.~S. Romanovski\u{\i}.
\newblock The occurrence problem for extensions of abelian groups by nilpotent
  groups.
\newblock {\em Siberian Mathematical Journal}, 21:273--276, 1980.
\newblock \href {https://doi.org/10.1007/BF00968275}
  {\path{doi:10.1007/BF00968275}}.

\bibitem{DBLP:journals/acta/Senizergues96}
G{\'{e}}raud S{\'{e}}nizergues.
\newblock On the rational subsets of the free group.
\newblock {\em Acta Informatica}, 33(3):281--296, 1996.
\newblock \href {https://doi.org/10.1007/s002360050045}
  {\path{doi:10.1007/s002360050045}}.

\bibitem{shepherdson1959reduction}
John~C Shepherdson.
\newblock The reduction of two-way automata to one-way automata.
\newblock {\em IBM Journal of Research and Development}, 3(2):198--200, 1959.
\newblock \href {https://doi.org/10.1147/rd.32.0198}
  {\path{doi:10.1147/rd.32.0198}}.

\bibitem{DBLP:journals/ita/Silva04}
Pedro~V. Silva.
\newblock Free group languages: Rational versus recognizable.
\newblock {\em RAIRO---Theoretical Informatics and Applications}, 38(1):49--67,
  2004.
\newblock \href {https://doi.org/10.1051/ita:2004003}
  {\path{doi:10.1051/ita:2004003}}.

\bibitem{silva2017automata}
Pedro~V. Silva.
\newblock An automata-theoretic approach to the study of fixed points of
  endomorphisms.
\newblock In Ventura~E. Gonz\'{a}lez-Meneses~J., Lustig~M., editor, {\em
  Algorithmic and Geometric Topics Around Free Groups and Automorphisms},
  Advanced Courses in Mathematics---CRM Barcelona, pages 1--42. Birkh\"{a}user,
  2017.
\newblock \href {https://doi.org/10.1007/978-3-319-60940-9_1}
  {\path{doi:10.1007/978-3-319-60940-9_1}}.

\bibitem{Weiss2015}
Armin Wei{\ss}.
\newblock {\em On the Complexity of Conjugacy in Amalgamated Products and HNN
  Extensions}.
\newblock PhD thesis, Institut f\"{u}r Formale Methoden der Informatik,
  Universit\"{a}t Stuttgart, 2015.

\bibitem{Wilf78}
Herbert~S. Wilf.
\newblock A circle-of-lights algorithm for the ``money-changing problem''.
\newblock {\em The American Mathematical Monthly}, 85(7):562--565, 1978.
\newblock \href {https://doi.org/10.2307/2320864} {\path{doi:10.2307/2320864}}.

\end{thebibliography}

\end{document}